\documentclass[12pt]{article}
\usepackage[margin=1in]{geometry}

\usepackage{amsthm,amsmath,amssymb,mathtools}
\newtheorem{theorem}{Theorem}[section]
\newtheorem{claim}{Claim}[theorem]
\newtheorem{lemma}[theorem]{Lemma}
\newtheorem{problem}[theorem]{Problem}
\newtheorem{proposition}[theorem]{Proposition}

\newtheorem{conjecture}[theorem]{Conjecture}
\theoremstyle{definition}
\newtheorem{definition}[theorem]{Definition}

\newtheorem{example}[theorem]{Example}

\newcommand{\bF}{\mathbb F}
\newcommand{\bR}{\mathbb R}

\newcommand{\bZ}{\mathbb Z}

\newcommand{\cA}{\mathcal{A}}
\newcommand{\cB}{\mathcal{B}}
\newcommand{\cC}{\mathcal{C}}

\newcommand{\cM}{\mathcal{M}}

\newcommand{\ep}{\varepsilon}
\DeclareMathOperator{\Aut}{Aut}

\DeclareMathOperator{\cl}{cl}

\DeclareMathOperator{\PG}{PG}

\DeclareMathOperator{\GF}{GF}

\newcommand{\del}{\!\setminus\!}

\usepackage{authblk, comment, tikz, enumitem, kbordermatrix}
\usetikzlibrary{arrows,matrix,positioning,decorations.markings}

\def\VR{\kern-\arraycolsep\strut\vrule &\kern-\arraycolsep}
\def\vr{\kern-\arraycolsep & \kern-\arraycolsep}

\newcommand*{\mm}{%
  \leavevmode
  \hphantom{0}%
  \llap{%
    \settowidth{\dimen0 }{$0$}%
    \resizebox{1.1\dimen0 }{\height}{$-$}%
  }%
}

\tikzstyle{every path}=[thick]

\newdimen\R
\title{Matroids from gain graphs over quotient groups}
\author{Zach Walsh\thanks{The author was supported in part by NSF grant DMS-2452015.}}
\affil{\small Auburn University, Department of Mathematics and Statistics, Auburn, AL, U.S.A.}
\date{\today}

\begin{document}

\maketitle

\begin{abstract}
We present a new construction for matroids from gain graphs that simultaneously generalizes several existing constructions.
The construction takes as input a gain graph over a Frobenius group $\Gamma$ with Frobenius kernel $\Gamma_1$ and outputs an elementary lift of the frame matroid of the underlying gain graph over the quotient group $\Gamma/\Gamma_1$.

While the hypothesis that $\Gamma$ is a Frobenius group may seem unusual, we prove that it is in some sense necessary: if $\Gamma$ is any finite group with a nontrivial proper normal subgroup $\Gamma_1$ and there is a construction that takes in a complete $\Gamma$-gain graph and outputs an elementary lift $M$ of the frame matroid of the underlying $(\Gamma/\Gamma_1)$-gain graph so that a cycle of the graph is a circuit of $M$ if and only if it is $\Gamma$-balanced, then $\Gamma$ is a Frobenius group with Frobenius kernel $\Gamma_1$.
\end{abstract}

\section{Introduction} \label{sec: introduction}

Given a group $\Gamma$, a \emph{$\Gamma$-gain graph} is a pair $(G, \psi)$ where $G$ is graph and $\psi$ is a function (called a \emph{$\Gamma$-gain function}) from the oriented edges of $G$ to $\Gamma$ so that reversing an edge inverts the associated value in $\Gamma$.
Classical work of Zaslavsky \cite{Zaslavsky1991} shows how to construct two different matroids on the edge set of $G$ from $(G, \psi)$, namely the \emph{frame matroid} $F(G, \psi)$ and the \emph{lift matroid} $L(G, \psi)$.
The lift matroid $L(G, \psi)$ is an \emph{elementary lift} of the graphic matroid $M(G)$ of $G$, meaning that there is a matroid $M$ with an element $e$ so that $M\del e = L(G, \psi)$ and $M/e = M(G)$.
Frame matroids and lift matroids have proven to be important in the study of hyperplane arrangements \cite{Anderson-Su-Zaslavsky2024}, rigidity theory \cite{Bernstein2022, Tanigawa-2015}, structural matroid theory \cite{Kahn-Kung1982}, combinatorial optimization \cite{Guenin-Stuive-2016}, and more \cite{Zaslavsky1998_survey}.

Our main result is a construction that has both of Zaslavsky's constructions, and also recent constructions of Anderson, Su, and Zaslavsky \cite{Anderson-Su-Zaslavsky2024} and Bernstein \cite{Bernstein2022}, as special cases.
We will work with gain graphs over a group $\Gamma$ with a normal subgroup $\Gamma_1$.
In this setting every $\Gamma$-gain graph $(G, \psi)$ has an underlying $(\Gamma/\Gamma_1)$-gain graph $(G, \psi/{\Gamma_1})$ via the natural homomorphism from $\Gamma$ to $\Gamma/\Gamma_1$.
Our construction takes as input a (possibly empty) collection $\cA$ of nontrivial subgroups of $\Gamma$ with the following properties:
\begin{itemize}[leftmargin=1.5em]
    \item $ \{\Gamma_1\} \cup \cA$ is a \emph{partition} of $\Gamma$ (every non-identity element of $\Gamma$ is in exactly one group in $\{\Gamma_1\} \cup \cA$),

    \item each subgroup $A \in \cA$ is \emph{malnormal} (for all $\gamma \in \Gamma - A$, the conjugate $\gamma^{-1}A\gamma$ intersects $A$ only in the identity), and 

    \item if $A \in \cA$, then every conjugate of $A$ is in $\cA$.
\end{itemize}
As we shall see in Section \ref{sec: group theory},  if $\Gamma$ is finite and $\Gamma_1$ is proper and nontrivial then these properties imply that
$\Gamma$ is a \emph{Frobenius group}, which is an important type of permutation group.
However, our main result, which we will now state, also applies to infinite groups such as the special Euclidean group of direct isometries of $\mathbb R^2$.

\begin{theorem} \label{thm: main}
Let $\Gamma$ be a group, let $\Gamma_1$ be a normal subgroup of $\Gamma$, and let $\{\Gamma_1\} \cup \cA$ be a partition of $\Gamma$ so that if $A \in \cA$ then $A$ is malnormal and every conjugate of $A$ is in $\cA$.
Then every $\Gamma$-gain graph $(G, \psi)$ has a canonical associated matroid $M(\Gamma, \Gamma_1, \cA, G, \psi)$ on the edge set of $G$, and this matroid is an elementary lift of the underlying $(\Gamma/\Gamma_1)$-frame matroid $F(G, \psi/\Gamma_1)$.
\end{theorem}

\noindent
Theorem \ref{thm: main} recovers several existing constructions for matroids from gain graphs:

\begin{itemize}[leftmargin=1.5em]
    \item If $\Gamma_1$ is trivial then $M(\Gamma, \Gamma_1, \cA, G, \psi)$ is the frame matroid of $(G, \psi)$.

    \item If $\Gamma_1 = \Gamma$ then $M(\Gamma, \Gamma_1, \cA, G, \psi)$ is the lift matroid of $(G, \psi)$.

    \item If $\Gamma$ is the special Euclidean group $SE(2)$ of direct isometries (transformations that preserve distance and orientation) of $\bR^2$, $\Gamma_1$ is the subgroup of translations, and $A \in \cA$ if $A$ is the stabilizing subgroup of some point $x \in \bR^2$, then $M(\Gamma, \Gamma_1, \cA, G, \psi)$ agrees with a recent construction of Bernstein \cite{Bernstein2022}, which was motivated in part by symmetry-forced rigidity of frameworks.
    We will discuss this example in detail in Section \ref{sec: rigidity}.

    \item If $\Gamma = \Gamma_1 A$ where $\Gamma_1$ is an abelian group with no elements of order $2$ and $A$ has order $2$ then there is a natural choice for $\cA$ (see Example \ref{ex: inverted infinite abelian group}), and $M(\Gamma, \Gamma_1, \cA, G, \psi)$ agrees with a recent construction of Anderson, Su, and Zaslavsky \cite{Anderson-Su-Zaslavsky2024} for matroids from gain signed graphs, which was motivated in part by the study of hyperplane arrangements. We will discuss this example in detail in Section \ref{sec: gain signed graphs}.
\end{itemize}

\noindent
Theorem \ref{thm: main} also has several interesting new special cases:

\begin{itemize}[leftmargin=1.5em]

    \item For any Frobenius group $\Gamma$, Theorem \ref{thm: main} applies with $\Gamma_1$ as the Frobenius kernel and $\cA$ as the set of all Frobenius complements.
    We will define these terms in Section \ref{sec: group theory}, and explore some interesting choices of Frobenius groups.

    \item Let $\bF$ be a field and 
let $\Gamma_1$ and $\Gamma_2$ be nontrivial subgroups of $\bF^+$ and $\bF^{\times}$, respectively, so that $\Gamma_1$ is closed under scaling by elements of $\Gamma_2$.
We will show in Theorem \ref{thm: matrix representation} that if $M$ is the vector matroid of a matrix over $\bF$ of the form
$$
\begin{tabular}{|ccc|}
\hline 
& a row with entries from $\Gamma_1$ & \\
\hline 
& & \\
& a $\Gamma_2$-frame matrix & \\ [0.5cm]
\hline
\end{tabular}
$$
then $M$ arises from Theorem \ref{thm: main}.
The special case $|\Gamma_2| = 1$ recovers a result of Zaslavsky \cite[Theorem 4.1]{Zaslavsky2003} and the special case $|\Gamma_2| = 2$ recovers a result of Anderson, Su, and Zaslavsky \cite[Theorem 4.2]{Anderson-Su-Zaslavsky2024}.
Here a \emph{$\Gamma_2$-frame matrix} is a matrix so that each column has at most two nonzero entries, each nonzero column contains a $1$, and each column with two nonzero entries contains a $1$ and a $-\gamma$ for some $\gamma \in \Gamma_2$.
It is expected that matrices of the form shown above will play an important role in future work in structural matroid theory (see \cite[page 24]{Geelen-Gerards-Whittle2013_structure_overview} and \cite[Theorem 4.2]{Structure_Theory}), and we hope that the construction of these matroids from gain graphs will provide useful tools for their future study.
\end{itemize}

While the inclusion of the partition $\cA$ in the hypothesis of Theorem \ref{thm: main} may seem unusual, we will prove that when $\Gamma$ is finite it is in fact necessary under one natural assumption, which we now describe.
Given a $\Gamma$-gain graph $(G, \psi)$, a cycle of $G$ is \emph{$\psi$-balanced} if there is a simple closed walk around $C$ so that if we multiply the gain values of the oriented edges along the walk we obtain the identity element of $\Gamma$.
If one seeks to construct a matroid $M$ from $(G, \psi)$, it is natural to require that a cycle of $G$ is a circuit of $M$ if and only if it is $\psi$-balanced.
Indeed, this is the case for the matroids of Theorem \ref{thm: main} and therefore also all of the previous constructions that it generalizes.
It is also natural to require that the construction applies to any $\Gamma$-gain graph, in particular to the \emph{complete $\Gamma$-gain graph $(K_n^{\Gamma}, \psi_n^{\Gamma})$}, which we will define in Section \ref{sec: gain graphs}.
The following result shows that if $M$ respects the balanced cycles of $(K_n^{\Gamma}, \psi_n^{\Gamma})$ as described above, then $M$ arises from Theorem \ref{thm: main}.

\begin{theorem} \label{thm: main converse}
Let $\Gamma$ be a finite group, let $\Gamma_1$ be a normal subgroup of $\Gamma$, and let $M$ be an elementary lift of the frame matroid $F(K_n^{\Gamma}, \psi_n^{\Gamma}/\Gamma_1)$ with $n \ge 4$ so that a cycle of $K_n^{\Gamma}$ is a circuit of $M$ if and only if it is $\psi_n^{\Gamma}$-balanced.
Then there is a partition $\{\Gamma_1\} \cup \cA$ of $\Gamma$ so that if $A \in \cA$ then $A$ is malnormal and every conjugate of $A$ is in $\cA$, and $M = M(\Gamma, \Gamma_1, \cA, K_n^{\Gamma}, \psi_n^{\Gamma})$.
\end{theorem}

In particular, if $\Gamma_1$ is proper and nontrivial, then $\Gamma$ is a Frobenius group with Frobenius kernel $\Gamma_1$ and set $\cA$ of Frobenius complements (see Proposition \ref{prop: Frobenius properties}).

The rest of this paper is structured as follows.
In Section \ref{sec: preliminaries} we will give some background information and notation for Frobenius groups, gain graphs, frame matroids, and lifted-graphic matroids.
We will also discuss how the structure of the partition $\{\Gamma_1\} \cup \cA$ when $\Gamma_1$ is proper and nontrivial leads to an interesting generalization of Frobenius groups to infinite groups, and state a problem (Problem \ref{prob: infinite group partition properties}) concerning properties of these groups.
We will prove Theorem \ref{thm: main} in Section \ref{sec: the construction}, and in Section \ref{sec: properties} we will prove some properties of the class of matroids of the form $M(\Gamma, \Gamma_1, \cA, G, \psi)$, including closure under minors when $\Gamma$ is finite (Theorem \ref{thm: minor-closed}).
In Section \ref{sec: examples} we will describe more details of the special cases of Theorem \ref{thm: main} listed above, and in Section \ref{sec: converse} we will prove Theorem \ref{thm: main converse}.
We will conclude in Section \ref{sec: future work} with a discussion on directions for future work.
We will assume familiarity with matroids, and we refer the reader to \cite{Oxley2011} for any undefined matroid terminology.

\section{Preliminaries} \label{sec: preliminaries}

In this section we will review some background on Frobenius groups,  and gain graphs and their associated matroids.

\subsection{Group theory} \label{sec: group theory}

We will first review the relevant group theory background.

\subsubsection{Semidirect products}
Many of the groups to which we will apply Theorem \ref{thm: main} will be defined using an operation called semidirect product that generalizes group direct products.
Let $\Gamma$ be a group with identity $1$. 
If $\Gamma$ has subgroups $\Gamma_1$ and $\Gamma_2$ so that $\Gamma_1 \cap \Gamma_2 = \{1\}$, $\Gamma_1$ is normal, and $\Gamma = \Gamma_1\Gamma_2$, then $\Gamma$ is an \emph{inner semidirect product} of $\Gamma_1$ and $\Gamma_2$, and we write $\Gamma = \Gamma_1 \rtimes \Gamma_2$.

\begin{example}
Let $D_{2n}$ be the dihedral group of order $2n$ with $n$ odd, let $R$ be the subgroup of rotations, and let $S$ be a subgroup consisting of the identity and one reflection $s$.
Since $R$ is a normal subgroup, $R$ and $S$ intersect only in the identity, and every reflection in $D_{2n}$ is of the form $rs$ for some $r \in R$, we have $D_{2n} = R \rtimes S$. 
\end{example}

\begin{example} \label{ex: Euclidean group as semidirect product}
Let $E(n)$ be the \emph{Euclidean group} of isometries (distance-preserving transformations) of $\bR^n$.
Let $T(n)$ be the subgroup of translations and let $O(n)$ be the \emph{orthogonal group}, which is the subgroup of isometries that fix the origin.
Clearly these two subgroups share only the identity transformation, and 
it is straightforward to show that $T(n)$ is a normal subgroup of $E(n)$.
If $f \in E(n)$ maps the origin to a vector $b$, then the map $x \mapsto f(x) - b$ is in $O(n)$.
Therefore every map in $E(n)$ can be uniquely written as the composition of a translation and an element of $O(n)$, and so $E(n) = T(n) \rtimes O(n)$.

Similarly, the \emph{special Euclidean group} $SE(n)$, which is the subgroup of $E(n)$ consisting of isometries that also preserve orientation (called \emph{direct isometries}), is a semidirect product $T(n)\rtimes SO(n)$ where $SO(n)$ is the \emph{special orthogonal group} of direct isometries that fix the origin.
\end{example}

With a slightly different perspective, one can construct groups that decompose as an inner semidirect product.
If $\Gamma = \Gamma_1 \rtimes \Gamma_2$, then there is a homomorphism $\phi$ from $\Gamma_2$ to $\Aut(\Gamma_1)$ (the group of automorphisms of $\Gamma_1$) defined by conjugation: for each $b \in \Gamma_2$ the function $\phi_b$ that maps $a \in \Gamma_1$ to $b^{-1}ab$ is an automorphism of $\Gamma_1$.
One can use this observation to combine two groups to form a new group that decomposes as an inner semidirect product.
Let $\Gamma_1$ and $\Gamma_2$ be groups (with $\Gamma_1$ additive and $\Gamma_2$ multiplicative) and let $\phi$ be a homomorphism from $\Gamma_2$ to $\Aut(\Gamma_1)$, writing $\phi_b$ for $\phi(b)$.
Then the group $\Gamma$ on the cartesian product of the sets underlying $\Gamma_1$ and $\Gamma_2$ with group operation $\circ$ defined by $(a, b) \circ (c, d) = (a + \phi_b(c), bd)$ is an \emph{outer semidirect product} of $\Gamma_1$ and $\Gamma_2$, and we write $\Gamma = \Gamma_1 \rtimes_{\phi} \Gamma_2$.
We may omit the subscript $\phi$ when $\phi$ is implicit.
It is straightforward to show that $\{(a, 1) \colon a \in \Gamma_1\}$ is a normal subgroup of $\Gamma$ isomorphic to $\Gamma_1$ and $\{(0, b) \colon b \in \Gamma_2\}$ is a subgroup of $\Gamma$ isomorphic to $\Gamma_2$, and that $\Gamma$ is an inner semidirect product of these subgroups.
The identity element of $\Gamma$ is $(0, 1)$, and the inverse of $(a, b)$ is $(\phi_{b^{-1}}(-a), b^{-1})$.

\begin{example} \label{ex: direct product}
Let $\Gamma_1$ and $\Gamma_2$ be any groups.
For all $b \in \Gamma_2$ let $\phi_b$ be the identity automorphism of $\Gamma_1$.
Then $\Gamma_1 \rtimes_{\phi} \Gamma_2$ is the direct product of $\Gamma_1$ and $\Gamma_2$.
\end{example}

\begin{example} \label{ex: semidirect product with {1,-1}}
Let $\Gamma_1$ be any (additively-written) group and let $\{1, -1\}^{\times}$ be the multiplicatively-written group of order $2$.
Let $\phi_1$ be the identity automorphism of $\Gamma_1$ and let $\phi_{-1}$ be the automorphism that inverts each element of $\Gamma_1$.
Then $\Gamma_1 \rtimes_{\phi} \{1, -1\}^{\times}$ is an outer semidirect product.
Note that for all $a \in \Gamma_1$, the element $(a, -1)$ of $\Gamma$ is its own inverse.
\end{example}

\begin{example} \label{ex: finite field semidirect product}
Let $\bF$ be a field, let $m$ be a positive integer, and let $\bF^m$ be the additive group of $m$-dimensional vectors over $\bF$.
For each $b \in \bF^{\times}$ define $\phi_b \colon \bF^m \to \bF^m$ by scaling, so $\phi_b(x) = bx$.
Then $\bF^m \rtimes_{\phi} \bF^{\times}$ is an outer semidirect product with $(a,b) \circ (c,d) = (a + bc, bd)$.
\end{example}

\subsubsection{Frobenius groups}
In Theorem \ref{thm: main}, if $\Gamma_1$ is trivial or equal to $\Gamma$, then $\Gamma$ can be any group.
However, when $\Gamma_1$ is nontrivial and proper, the existence of the collection $\cA$ of nontrivial malnormal subgroups of $\Gamma$ places significant structure on $\Gamma$.
In particular, when $\Gamma$ is finite, it implies that $\Gamma$ is a special finite permutation group called a Frobenius group.

Following \cite{Frobenius-groups}, a \emph{Frobenius group} is a finite group $\Gamma$ with a proper nontrivial malnormal subgroup $A$, called a \emph{Frobenius complement}.
A remarkable result of Frobenius shows that the set $\Gamma_1 = (\Gamma - (\cup_{g \in \Gamma} \, g^{-1}Ag)) \cup \{1\}$ is always a nontrivial normal subgroup of $\Gamma$, referred to as a \emph{Frobenius kernel} of $\Gamma$.
Therefore every Frobenius group has the structure of the groups of Theorem \ref{thm: main} with $\cA$ as the set of conjugates of a Frobenius complement, and conversely, if $\cA$ contains a proper subgroup, then $\Gamma$ is a Frobenius group (when $\Gamma$ is finite).
Moreover, it is well-known (see \cite[Theorem 4.2]{Frobenius-groups}) that the Frobenius kernel is unique and that any two Frobenius complements are conjugate, so a group is Frobenius in at most one way.

Frobenius groups can also be defined as special permutation groups.
Let $\Gamma$ be a transitive group of permutations of a set $X$ (for all $x,y \in X$ there is a permutation sending $x$ to $y$) so that each non-identity permutation has at most one fixed point, and some non-identity permutation has a fixed point.
It is straightforward to show that for each $x \in X$, the subgroup of permutations that fix $x$ (the \emph{stabilizer} of $x$) is malnormal, and that any two stabilizers are conjugate.
Therefore if $\Gamma$ is finite then it is a Frobenius group.
Here the Frobenius kernel is the subgroup of fixed-point-free permutations (called \emph{derangements}).
It turns out that every Frobenius group arises in this way by taking $X$ to be the set of right cosets of a Frobenius complement (see \cite[Theorem 1.7]{Frobenius-groups}).
The permutation perspective is often natural for describing examples.

\begin{example} \label{ex: dihedral group}
Let $X$ be the set of vertices of a regular $n$-gon with $n$ odd, and let $D_{2n}$ be the dihedral group of order $2n$, viewed as a group of permutations of $X$.
Each nontrivial rotation has no fixed points, and each nontrivial reflection has one fixed point since $n$ is odd.
Therefore $D_{2n}$ is a Frobenius group where the kernel is the group of rotations and each reflection with the identity is a complement.
\end{example}

\begin{example} \label{ex: finite field affine transformations}
For each finite field $\bF$ with at least three elements and each positive integer $m$, the group of affine linear transformations $x \mapsto a + bx$ of $\bF^m$ with $b \ne 0$ is a Frobenius group with $X = \bF^m$.
The kernel is the subgroup of translations and the complements are the stabilizers.
\end{example}

In the following example it is more natural to simply identify a proper nontrivial malnormal subgroup.

\begin{example} \label{ex: inverted finite odd abelian group}
Let $\Gamma_1$ be an (additively-written) nontrivial finite abelian group with odd order, and let $\Gamma = \Gamma_1 \rtimes \{1,-1\}^{\times}$ where $-1$ acts by inversion, as in Example \ref{ex: semidirect product with {1,-1}}.
It is straightforward to show that for each $a \in \Gamma_1 - \{0\}$ the subgroup $\{(0,1), (a, -1)\}$ in $\cA$ is malnormal (we show this in a more general setting in Example \ref{ex: inverted infinite abelian group}), so $\Gamma$ is a Frobenius group.
The kernel is the subgroup isomorphic to $\Gamma_1$ and the complements are the subgroups $\{(0,1), (a, -1)\}$ with $a \in \Gamma_1 - \{0\}$.
When $\Gamma_1$ is cyclic we recover Example \ref{ex: dihedral group} and when $\Gamma_1 = \GF(3)^m$ we recover Example \ref{ex: finite field affine transformations} for $\bF = \GF(3)$.
\end{example}

More exotic examples of Frobenius groups can be found in \cite{Starostin-1971}, and more can be constructed from the following fact: if $\Gamma_1 \rtimes \Gamma_2$ and $\Gamma_1' \rtimes \Gamma_2$ are Frobenius groups, then so is $(\Gamma_1 \times \Gamma_1') \rtimes \Gamma_2$, where $\Gamma_1 \times \Gamma_1'$ is the direct product of $\Gamma_1$ and $\Gamma_1'$.

We can use properties of Frobenius groups to state some useful properties of the groups of Theorem \ref{thm: main} when they are finite.
A group partition is \emph{nontrivial} if it contains at least two nontrivial subgroups.

\begin{proposition} \label{prop: Frobenius properties}
Let $\Gamma$ be a group with a normal subgroup $\Gamma_1$ and a nontrivial partition $\{\Gamma_1\} \cup \cA$ so that if $A \in \cA$ then $A$ is malnormal and every conjugate of $A$ is in $\cA$.
If $\Gamma$ is finite then the following hold:
\begin{enumerate}[label=$(\roman*)$]

\item The pair $(\Gamma_1, \cA)$ is unique with $\Gamma_1$ as the Frobenius kernel and $\cA$ as the set of Frobenius complements.

\item Any two groups in $\cA$ are isomorphic.

\item Any two groups in $\cA$ are conjugate.

\item There is a group $A \in \cA$ so that $\Gamma = \Gamma_1 \rtimes A$.

\item $\Gamma_1$ is nontrivial.
\end{enumerate}
\end{proposition}
\begin{proof}
Every proper nontrivial malnormal subgroup of $\Gamma$ is a Frobenius complement, so every group in $\cA$ is a Frobenius complement.
It is well-known that any two Frobenius complements are conjugate (see \cite[Theorem 4.2]{Frobenius-groups}), so $(ii)$ and $(iii)$ hold.
Since $\cA$ is closed under group conjugation, it follows that $\cA$ is the set of Frobenius complements and $\Gamma_1$ is the Frobenius kernel of $\Gamma$, which is also known to be unique (see \cite[Theorem 4.2]{Frobenius-groups}), so $(i)$ holds.
It is also well-known that $(iv)$ holds for any $A \in \cA$ (see \cite[Theorem 1.2]{Frobenius-groups}).
Finally, each $A \in \cA$ is a proper subgroup because the partition is nontrivial, so $(iv)$ implies that $\Gamma_1$ is nontrivial.
\end{proof}

\subsubsection{Group partitions}
The notion of a group partition will help us describe the groups for which Theorem \ref{thm: main} applies.
A \emph{partition} of a group $\Gamma$ is a collection $\Pi$ of subgroups of $\Gamma$ so that each non-identity element of $\Gamma$ is in exactly one subgroup in $\Pi$ \cite{Miller-1906}.
The partition is \emph{nontrivial} if $\Pi$ contains at least two nontrivial groups.
For example, every Frobenius group has a nontrivial partition into its Frobenius kernel and Frobenius complements.
Non-Frobenius groups can also have a nontrivial partition.
For example, the direct sum of at least two copies of a cyclic group of prime order is not Frobenius but has a nontrivial partition into its cyclic subgroups.
Amazingly, work of Baer \cite{Baer-1961}, Kegel \cite{Kegel-1961}, and Suzuki \cite{Suzuki-1961} combines to completely classify the finite groups with a nontrivial partition into six families, one of which is the family of Frobenius groups.
We make the following definition to describe the partition from Theorem \ref{thm: main}.

\begin{definition}  \label{def: Frobenius partition}
Let $\Gamma$ be a (possibly infinite) group.
A partition $\{\Gamma_1\} \cup \cA$ of $\Gamma$ is a \emph{Frobenius partition} if $\Gamma_1$ is a normal subgroup of $\Gamma$, and if $A \in \cA$ then $A$ is malnormal and every conjugate of $A$ is in $\cA$.
A Frobenius partition is \emph{nontrivial} if it contains at least two nontrivial groups, and is \emph{trivial} otherwise.
\end{definition}

Note that every group $\Gamma$ has trivial Frobenius partition $\{\Gamma\}$.
And if $\Gamma$ is a Frobenius group with kernel $\Gamma_1$ and set $\cA$ of complements, then $\{\Gamma_1\} \cup \cA$ is a nontrivial Frobenius partition of $\Gamma$.
Conversely, by Proposition \ref{prop: Frobenius properties}, every finite group with a nontrivial Frobenius partition is a Frobenius group.
We will next discuss properties of infinite groups with a nontrivial Frobenius partition, and see that there are many natural groups of this form.
This gives a seemingly new generalization of Frobenius groups to infinite groups, and leads to some problems (Problem \ref{prob: infinite group partition properties}) in group theory that seem nontrivial and interesting.

\subsubsection{Infinite groups}
When $\Gamma$ is infinite with a proper nontrivial malnormal subgroup $A$, it is not guaranteed that $(\Gamma - (\cup_{g \in \Gamma} \, g^{-1}Ag)) \cup \{1\}$ is a subgroup (see \cite[pg. 51]{Kegel-Wehrfritz-1973} or Example \ref{ex: inverted infinite abelian group}), unless $\Gamma$ has special properties \cite{Collins-1990, Kegel-Wehrfritz-1973}.
However, in many natural examples either this set is a normal subgroup of $\Gamma$, or we can find a nontrivial Frobenius partition of $\Gamma$ in a different way.
The first two examples will use the previously discussed fact that if $\Gamma$ acts transitively on a set so that each nontrivial permutation has at most one fixed point and some nontrivial permutation has a fixed point, then every stabilizer subgroup of $\Gamma$ is malnormal and any two stabilizers are conjugate.

\begin{example} \label{ex: infinite field affine transformations}
Example \ref{ex: finite field affine transformations} generalizes directly to infinite fields.
Let $\bF$ be a field with at least three elements, let $m$ be a positive integer, and let $\Gamma$ be the group of affine transformations $x \mapsto a + bx$ of $\bF^m$ with $b \ne 0$.
It is straightforward to check that each transformation has at most one fixed point, and that a transformation has no fixed points if and only if it is a non-identity translation $x \mapsto a + x$ with $a \ne 0$.
Therefore each point-stabilizing subgroup is malnormal, and any two point stabilizers are conjugate.
It is straightforward to check that the subgroup of translations is a normal subgroup of $\Gamma$, even when $\bF$ is infinite.
Therefore $\Gamma$ has a Frobenius partition $\{\Gamma_1\} \cup \cA$ where $\Gamma_1$ is the subgroup of translations and $\cA$ is the set of point-stabilizing subgroups of $\Gamma$.
We note that $x \mapsto a + bx$ and $x \mapsto c + dx$ with $b,d \ne 1$ are in the same set in $\cA$ if and only if  $a/(1 - b) = c/(1 - d)$.
We also comment that $\Gamma \cong \bF^m \rtimes \bF^{\times}$ where $\bF^{\times}$ acts by scaling.
Here the element $(a, b)$ corresponds to the map $x \mapsto a + bx$.
\end{example}

\begin{example} \label{ex: special Euclidean group as Frobenius group}
Recall from Example \ref{ex: Euclidean group as semidirect product} that the special Euclidean group $SE(n)$ of direct isometries of $\bR^n$ is a semidirect product $T(n) \rtimes SO(n)$ where $T(n)$ is the normal subgroup of translations of $\bR^n$ and $SO(n)$ is the special orthogonal group of rotations of $\bR^n$.
When $n = 2$, each transformation has at most one fixed point, and a transformation has no fixed points if and only if it is a non-identity translation.
(To see this, note that if a matrix $A$ rotates by angle $\theta$ then $A$ has eigenvalues $e^{i\theta}$ and $e^{-i\theta}$, so if $\theta$ is not a multiple of $2\pi$ then $A - I$ is invertible, and so the map $x \mapsto Ax + b$ only fixes $x = (A - I)^{-1}(-b)$.)
Therefore each point-stabilizing subgroup of $SE(2)$ is malnormal, and any two stabilizers are conjugate.
Then $SE(2)$ has Frobenius partition $\{T(2)\} \cup \cA$ where $\cA$ is the set of point-stabilizing subgroups.
\end{example}

As we shall see in the next example, if $\Gamma$ is infinite with Frobenius partition $\{\Gamma_1\} \cup \cA$, it may be the case that the groups in $\cA$ are not pairwise conjugate; by Proposition \ref{prop: Frobenius properties} this is not possible when $\Gamma$ is finite.

\begin{example} \label{ex: inverted infinite abelian group}
We can generalize Example \ref{ex: inverted finite odd abelian group} to infinite groups. Let $\Gamma_1$ be an (additively-written) nontrivial abelian group with no elements of order $2$, and let $\Gamma = \Gamma_1 \rtimes \{1,-1\}^{\times}$ where $-1$ acts by inversion, as in Example \ref{ex: semidirect product with {1,-1}}.
Recall that for all $a \in \Gamma_1$, the element $(a, -1)$ of $\Gamma$ is its own inverse.
Let 
$$\cA = \Big\{\{(0,1), (a, -1)\} \colon a \in \Gamma_1\Big\}.$$
We claim that each group $\{(0,1), (a, -1)\}$ in $\cA$ is malnormal.
To see this, note the following two calculations for $c \in \Gamma_1$:
\begin{align*}
(c, 1)^{-1} \circ \{(0,1), (a, -1)\} \circ (c, 1) &= \{(0, 1), (-c + a - c, -1)\} \\ 
(c, -1)^{-1} \circ \{(0,1), (a, -1)\} \circ (c, -1) &= \{(0, 1), (c - a + c, -1)\}.
\end{align*}
Since $\Gamma_1$ is abelian with no elements of order $2$, we see that $-c + a - c = a$ if and only if $c = 0$, and $c - a + c = a$ if and only if $a = c$.
It follows that the subgroup $\{(0,1), (a, -1)\}$ is malnormal, and it is straightforward to check that all of its conjugates are in $\cA$.
Since $\Gamma_1$ is a normal subgroup we see that $\{\Gamma_1\} \cup \cA$ is a Frobenius partition of $\Gamma$.
From the calculations above, one can check that $\{(0,1), (a, -1)\}$ and $\{(0, 1), (b, -1)\}$ are conjugate if and only if there is some $c \in \Gamma_1$ with $a - b = 2c$ or $a + b = 2c$.
If every element of $\Gamma_1$ has a square root (for example, if $\Gamma_1$ is the additive group of a field), then this is always the case, so any two groups in $\cA$ are conjugate.
For other choices of $\Gamma_1$ this may not be true.
For example, if $\Gamma_1 = \bZ$, then $\cA$ contains two conjugacy classes based on the parity of $a$.
\end{example}

Example \ref{ex: inverted infinite abelian group} shows that property $(iii)$ of Proposition \ref{prop: Frobenius properties} does not always hold when $\Gamma$ is infinite.
Interestingly, property $(v)$ also does not always hold when $\Gamma$ is infinite.
As described in \cite[Example 5.2]{delaHarpe-Weber-2014}, there is an infinite group with a proper malnormal subgroup whose conjugates cover the group.
To the best of the author's knowledge, it is unknown whether properties $(i)$, $(ii)$, and $(iv)$ necessarily hold when $\Gamma$ is infinite.

\begin{problem} \label{prob: infinite group partition properties}
Let $\Gamma$ be an infinite group with a normal subgroup $\Gamma_1$ and a nontrivial partition $\{\Gamma_1\} \cup \cA$ so that if $A \in \cA$ then $A$ is malnormal and every conjugate of $A$ is in $\cA$. 
Determine which of the following properties necessarily hold:
\begin{enumerate}[label=$(\roman*)$]

\item The pair $(\Gamma_1, \cA)$ is unique.

\item Any two groups in $\cA$ are isomorphic.

\item There is a group $A \in \cA$ so that $\Gamma = \Gamma_1 \rtimes A$.
\end{enumerate}
\end{problem}

This seems like a nontrivial problem in the theory of infinite groups, and indicates that groups with a Frobenius partition form an interesting generalization of Frobenius groups to infinite groups.
While infinite generalizations of Frobenius groups are well-studied (see \cite{Collins-1990, delaHarpe-Weber-2014, Sozutov-Sunkov-1976}), it appears that the generalization given by Frobenius partitions is new.
Parts $(i)$ and $(iii)$ are also seemingly unknown if we require that $\cA$ is the set of conjugates of a proper nontrivial malnormal subgroup.
Placing this additional assumption on the Frobenius partition $\{\Gamma_1\} \cup \cA$ would give a different generalization of Frobenius groups that may also be interesting for future study.

\subsection{Gain graphs} \label{sec: gain graphs}
Gain graphs are graphs with edges invertibly labeled by elements of a group.
They have found applications in many areas; we direct the reader to Zaslavsky's comprehensive survey \cite{Zaslavsky1998_survey} for details.
In this section we will give background on gain graphs as needed for this paper, and give new notation for gain graphs over quotient groups.

Let $G$ be a graph.
An \emph{oriented edge} of $G$ is an ordered triple $(e,u,v)$ where $e$ is an edge with ends $u$ and $v$.
We write $\vec E(G)$ for the set of all oriented edges of $G$.
Let $G$ be a graph and let $\Gamma$ be a group (written multiplicatively).
A \emph{$\Gamma$-gain function} for $G$ is a function $\psi \colon \vec E(G) \to \Gamma$ so that $\psi(e, u, v) = \psi(e, v, u)^{-1}$ for every oriented edge $(e, u, v)$ with $u \ne v$, so reversing the orientation inverts the gain value.
The pair $(G, \psi)$ is a \emph{$\Gamma$-gain graph}.
Every walk in $G$ has an associated value in $\Gamma$.
For each walk $W = v_1, e_1, v_2, e_2, \dots, v_{k}, e_{k}, v_{k+1}$, where each $v_i$ is a vertex and each $e_i$ is an edge with ends $v_i$ and $v_{i+1}$, define $\psi(W) = \prod_{i = 1}^k \psi(e_i, v_i, v_{i+1})$.
We say that a cycle $C$ of $G$ is \emph{$\psi$-balanced} if there is a simple closed walk $W$ on $C$ for which $\psi(W) = 1$.
We write $\cB_{\psi}$ for the set of $\psi$-balanced cycles of $(G, \psi)$.
More generally we say that a set of edges of $G$ is \emph{$\psi$-balanced} is all of its cycles are $\psi$-balanced.

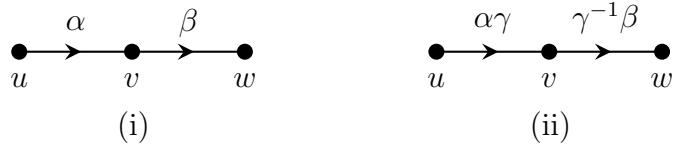
\begin{figure}
\centering
\setlength\tabcolsep{1cm}
\begin{tabular}{c c}
\begin{tikzpicture}[scale=0.75,line cap=round,line join=round,>=triangle 45,x=1cm,y=1cm, decoration={markings, 
	mark= at position 0.55 with {\arrow[scale = 1.5]{stealth}}}]

\draw[postaction={decorate}] (-2,0)-- (0,0);
\draw[postaction={decorate}]  (0,0)-- (2,0);

\draw [fill=black] (-2,0) circle (3.5pt);
\draw [fill=black] (0,0) circle (3.5pt);
\draw [fill=black] (2,0) circle (3.5pt);

\draw[color=black] (-2, -0.5) node {$u$};
\draw[color=black] (0, -0.5) node {$v$};
\draw[color=black] (2, -0.5) node {$w$};

\draw[color=black] (-1, 0.5) node {$\alpha$};
\draw[color=black] (1, 0.5) node {$\beta$};
\end{tikzpicture}

&

\begin{tikzpicture}[scale=0.75,line cap=round,line join=round,>=triangle 45,x=1cm,y=1cm, decoration={markings, 
	mark= at position 0.55 with {\arrow[scale = 1.5]{stealth}}}]

\draw[postaction={decorate}] (-2,0)-- (0,0);
\draw[postaction={decorate}]  (0,0)-- (2,0);

\draw [fill=black] (-2,0) circle (3.5pt);
\draw [fill=black] (0,0) circle (3.5pt);
\draw [fill=black] (2,0) circle (3.5pt);

\draw[color=black] (-2, -0.5) node {$u$};
\draw[color=black] (0, -0.5) node {$v$};
\draw[color=black] (2, -0.5) node {$w$};

\draw[color=black] (-1, 0.5) node {$\alpha  \gamma$};
\draw[color=black] (1, 0.6) node {$\gamma^{-1} \beta$};
\end{tikzpicture}

\\

(i) & (ii)
\end{tabular}
       
    \caption{Image (i) is a $\Gamma$-gain graph, and (ii) is the $\Gamma$-gain graph obtained by switching at $v$ with value $\gamma$.}
    \label{fig: switching}
\end{figure}

There is a natural operation on a $\Gamma$-gain function that preserves the set of balanced cycles.
Let $\psi$ be a $\Gamma$-gain function on a graph $G$, and let $\eta \colon V(G) \to \Gamma$.
Define a new $\Gamma$-gain function $\psi^{\eta}$ on $G$ by $\psi^{\eta}(e, u, v) = \eta(u)^{-1} \psi(e, u, v) \eta(v)$ for each oriented edge $(e, u, v)$ of $G$.
Then $\cB_{\psi^{\eta}} = \cB_{\psi}$, because if $W$ is a walk on $G$ then $\psi(W) = 1$ if and only if $\psi^{\eta}(W) = 1$.
We say that $\Gamma$-gain functions $\psi$ and $\psi'$ are \emph{switching equivalent} if there is a function $\eta \colon V(G) \to \Gamma$ so that $\psi' = \psi^{\eta}$.
We say that $\eta$ is a \emph{switching function}.
Note that if $e$ is a loop at vertex $v$, then $\psi^{\eta}(e, v, v)$ is simply the conjugation $\eta(v)^{-1}\psi(e, v, v)\eta(v)$ of $\psi(e, v, v)$.
Also note that if $W$ is a walk from $u$ to $v$, then $\psi^{\eta}(W) = \eta(u)^{-1}\psi(W)\eta(v)$.
In particular, if $u = v$ and $\eta(u)$ is the identity, then $\psi^{\eta}(W) = \psi(W)$.

A set $X$ of edges of $G$ is \emph{$\psi$-normalized} 
if each orientation of an edge in $X$ has the identity label under $\psi$.
We will make repeated use of the following modification of a standard result (see \cite[Proposition 2.1]{Funk-Pivotto-Slilaty-2022}, for example), which shows that any forest of $G$ can be normalized by switching.

\begin{lemma} \label{lem: normalize a forest}
Let $\Gamma$ be a group with identity $\ep$, let $(G, \psi)$ be a $\Gamma$-gain graph, let $F$ be a forest of $G$, and let $v$ be a vertex of $G$.
Then there is a switching function $\eta$ so that the forest $F$ is $\psi^{\eta}$-normalized and $\eta(v) = \ep$.
\end{lemma}
\begin{proof}
We proceed by induction on $|V(G)|$.
The statement is clearly true if $|V(G)| = 1$, so we may assume that $|V(G)| \ge 2$.
If there is some $w \in V(G) - V(F)$ (possibly $w = v$) then the statement holds by applying induction to $G - w$, so we may assume that $V(F) = V(G)$.
Let $v \in V(F)$ and let $F' = F - v$.
By induction, there is a switching function $\eta'$ so that $F'$ is $\psi^{\eta'}$-normalized and $\eta'(v) = \ep$.
If $v$ has no incident edges in $F$ then $F$ is also $\psi^{\eta'}$-normalized, so we may assume that $v$ has nonzero degree in $F$.
Let $u_1, u_2, \dots, u_k$ be the neighbors of $v$ in $F$.
For each $i \in [k]$ let $e_i$ be the edge of $F$ with ends $v$ and $u_i$,  let $\alpha_i = \psi^{\eta'}(e_i, u_i, v_i)$, and let $F_i$ be the component of $F'$ that contains vertex $u_i$. 
Let $\eta''$ be the switching function with $\eta''(w) = \alpha_i$ if $w \in V(F_i)$ and $\eta''(w) = \ep$ otherwise.
Let $\eta$ be the switching function that first applies $\eta'$ and then applies $\eta''$.
Then $F$ is $\psi^{\eta}$-normalized and $\psi^{\eta}(v) = \ep$.
\end{proof}

We will use some nonstandard notation to keep track of the set of gain values of the two orientations of a given edge $e$.

\begin{definition} \label{def: set of gain values}
Let $\Gamma$ be a group, let $(G, \psi)$ be a $\Gamma$-gain graph, and let $e$ be an edge of $G$ with ends $u$ and $v$.
We write $\psi(e)$ for the set $\{\psi(e, u, v), \psi(e, v, u)\}$.
More generally, for a set $X \subseteq E(G)$ we write $\psi(X)$ for $\cup_{e \in X} \psi(e)$.
\end{definition}

Note that for each $e \in E(G)$, either $|\psi(e)| = 1$ and the element in $\psi(e)$ is its own inverse in $\Gamma$, or $|\psi(e)| = 2$ and the two elements in $\psi(e)$ are inverses in $\Gamma$.
And a set $X$ is $\psi$-normalized if and only if $\psi(X)$ contains only the identity of $\Gamma$.

When $\Gamma$ is finite there is a natural notion of a complete $n$-vertex $\Gamma$-gain graph.
For each finite group $\Gamma$ and integer $n \ge 2$ let $K_n^{\Gamma}$ be the graph with vertex set $\{v_i \colon i \in [n]\}$ and edge set $\binom{[n]}{2} \times \Gamma$ where the edge $(\{i,j\}, \alpha)$ has ends $v_i$ and $v_j$, and let $\psi_n^{\Gamma}$ be the $\Gamma$-gain function for which $\psi((\{i,j\}, \alpha), v_i, v_j) = \alpha$  when $i < j$.
We will often write $\alpha_{ij}$ for the edge $(\{i,j\}, \alpha)$.

We close with some notation for a gain function induced by a quotient group.

\begin{definition} \label{def: quotient group induced gain function}
Let $\Gamma$ be a group with a normal subgroup $\Gamma_1$ and let $(G,\psi)$ be a $\Gamma$-gain graph.
We write $\psi/{\Gamma_1}$ for the $(\Gamma/\Gamma_1)$-gain function whose value at an oriented edge $(e,u,v)$ is the image of $\psi(e,u,v)$ under the natural homomorphism from $\Gamma$ to $\Gamma/\Gamma_1$. 
\end{definition}

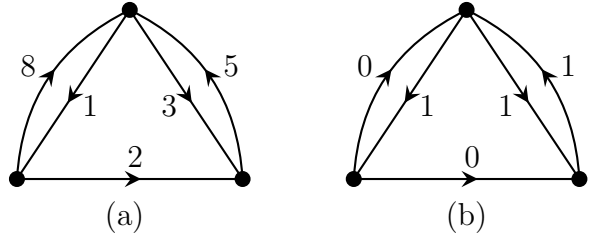
\begin{figure}
    \centering
\begin{tabular}{c ccc}
\setlength\tabcolsep{5cm}

\begin{tikzpicture}[scale=0.75,line cap=round,line join=round,>=triangle 45,x=1cm,y=1cm, decoration={markings, 
	mark= at position 0.55 with {\arrow[scale = 1.5]{stealth}}}]

\draw[postaction={decorate}] (0,1)-- (-2,-2);
\draw[postaction={decorate}]  (0,1)-- (2,-2);
\draw[postaction={decorate}]  (-2,-2)-- (2,-2);
\draw[postaction={decorate}]  (-2,-2) to[bend left=30] (0,1);
\draw[postaction={decorate}]  (2,-2) to[bend right=30] (0,1);

\begin{scriptsize}
\draw [fill=black] (0,1) circle (3.5pt);
\draw [fill=black] (-2,-2) circle (3.5pt);
\draw [fill=black] (2,-2) circle (3.5pt);
\draw[color=black] (-1.8, 0) node {\normalsize $8$};
\draw[color=black] (1.8, 0) node {\normalsize $5$};
\draw[color=black] (-0.7, -0.7) node {\normalsize $1$};
\draw[color=black] (0.7, -0.7) node {\normalsize $3$};
\draw[color=black] (0.1,-1.6) node {\normalsize $2$};
\end{scriptsize}

\end{tikzpicture}

& & &

\begin{tikzpicture}[scale=0.75,line cap=round,line join=round,>=triangle 45,x=1cm,y=1cm, decoration={markings, 
	mark= at position 0.55 with {\arrow[scale = 1.5]{stealth}}}]

\draw[postaction={decorate}] (0,1)-- (-2,-2);
\draw[postaction={decorate}]  (0,1)-- (2,-2);
\draw[postaction={decorate}]  (-2,-2)-- (2,-2);
\draw[postaction={decorate}]  (-2,-2) to[bend left=30] (0,1);
\draw[postaction={decorate}]  (2,-2) to[bend right=30] (0,1);

\begin{scriptsize}
\draw [fill=black] (0,1) circle (3.5pt);
\draw [fill=black] (-2,-2) circle (3.5pt);
\draw [fill=black] (2,-2) circle (3.5pt);
\draw[color=black] (-1.8, 0) node {\normalsize $0$};
\draw[color=black] (1.8, 0) node {\normalsize $1$};
\draw[color=black] (-0.7, -0.7) node {\normalsize $1$};
\draw[color=black] (0.7, -0.7) node {\normalsize $1$};
\draw[color=black] (0.1,-1.6) node {\normalsize $0$};
\end{scriptsize}

\end{tikzpicture}

\\
(a) &&& (b)
\end{tabular}
    \caption{Image (a) shows a $\bZ$-gain graph $(G, \psi)$ with one balanced cycle, and (b) shows the $(\bZ/2\bZ)$-gain graph $(G, \psi/2\bZ)$, which has three balanced cycles.}
    \label{fig: quotient gain function}
\end{figure}

See Figure \ref{fig: quotient gain function} for an example.
Note that if $X = \{e \in E(G) \colon \psi(e) \subseteq \Gamma_1\}$, then $X$ is $(\psi/\Gamma_1)$-normalized because each element of $\Gamma_1$ maps to the identity element of $\Gamma/\Gamma_1$ under the natural homomorphism from $\Gamma$ to $\Gamma/\Gamma_1$.
Also, in the special case that $\Gamma = \Gamma_1 \rtimes \Gamma_2$, each element of $\Gamma$ is of the form $(a, b)$ with $a \in \Gamma_1$ and $b \in \Gamma_2$, and if $\psi(e, u, v) = (a, b)$ then $(\psi/\Gamma_1)(e, u, v) = b$.

\subsection{Biased graphs and their matroids}

Every gain graph has an associated \emph{biased graph}, which is a pair $(G, \cB)$ where $G$ is a graph and $\cB$ is a set of cycles of $G$ so that no theta subgraph (see Figure \ref{fig: handcuffs}) contains exactly two cycles in $\cB$.
This property of $\cB$ is known as the \emph{theta property}.
The cycles in $\cB$ are \emph{balanced}, and the cycles not in $\cB$ are \emph{unbalanced}.
In his original paper introducing biased graphs, Zaslavsky \cite{Zaslavsky1989} showed that if $(G, \psi)$ is a $\Gamma$-gain graph and $\cB_{\psi}$ is the set of $\psi$-balanced cycles, then $(G, \cB_{\psi})$ is a biased graph.
In his seminal follow-up paper \cite{Zaslavsky1991} he showed how to define two matroids from a given biased graph, the frame matroid and the lift matroid, whose circuits are particular types of biased subgraphs of $(G, \cB)$.
A \emph{tight handcuff} is a graph consisting of two edge-disjoint cycles with exactly one common vertex, and a \emph{loose handcuff} is a graph consisting of two vertex-disjoint cycles with a minimal path connecting them.
Equivalently, they are subdivisions of the graphs shown in Figure \ref{fig: handcuffs}.

\begin{figure}
    \centering
\begin{tabular}{ccc ccc c}
 \begin{tikzpicture}[scale=0.75]

\begin{scope}
    \draw (0,0) circle (1);
    \fill (-1,0) circle (3pt);
    \fill (1,0) circle (3pt);
    \draw (-1,0) -- (1,0);
\end{scope}

\end{tikzpicture}

&&&
 \begin{tikzpicture}[scale=0.75]

\begin{scope}
    \draw (0,0) circle (1);
    \draw (2,0) circle (1);
    \fill (1,0) circle (3pt);
\end{scope}

\end{tikzpicture}
     &&&
 \begin{tikzpicture}[scale=0.75]

\begin{scope}
    \draw (0,0) circle (1);
    \draw (3,0) circle (1);
    \fill (1,0) circle (3pt);
    \fill (2,0) circle (3pt);
    \draw (1,0) -- (2,0);
\end{scope}

\end{tikzpicture}

\\
(a) &&& (b) &&& (c)

\end{tabular}
       
    \caption{Theta graphs, tight handcuffs, and loose handcuffs are subdivisions of the graphs (a), (b), and (c), respectively.}
    \label{fig: handcuffs}
\end{figure}
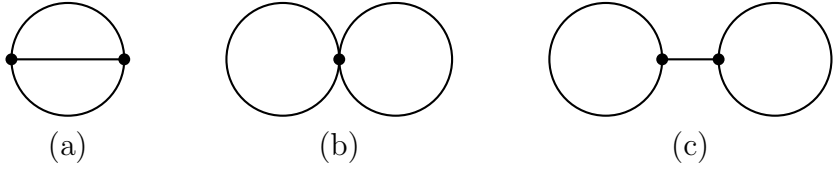

The \emph{frame matroid} of a biased graph $(G, \cB)$, denoted $F(G, \cB)$, is the matroid on $E(G)$ whose circuits are the balanced cycles, the tight handcuffs with both cycles unbalanced, the theta graphs with all cycles unbalanced, and the loose handcuffs with both cycles unbalanced.
In particular, if every cycle is balanced then the circuits of $F(G, \cB)$ are precisely the cycles of $G$, so $F(G, \cB)$ is equal to the graphic matroid of $G$.
It will be useful for us to recall the rank function of $F(G, \cB)$ (\cite[Theorem 2.1]{Zaslavsky1991}).
For $X \subseteq E(G)$ we write $G[X]$ for the graph with edge set $X$ and vertex set consisting of the vertices incident with an edge in $X$, and we say that a subgraph of $G$ is balanced if all of its cycles are balanced.

\begin{proposition}[{\cite{Zaslavsky1991}}]\label{prop: rank of X in frame matroid}
Let $(G, \cB)$ be a biased graph and let $X$ be a subset of $E(G)$. Then the rank of $X$ in $F(G, \cB)$ is given by $r(X)=|V(G[X])| - b(X)$, where $b(X)$ is the number of balanced components of $G[X]$.
\end{proposition}

For a gain graph $(G, \psi)$ we will write $F(G, \psi)$ for $F(G, \cB_{\psi})$.
For every group $\Gamma$, a \emph{$\Gamma$-frame matroid} is the frame matroid of a biased graph arising from a $\Gamma$-gain graph.
The class of frame matroids is minor-closed, as is the class of $\Gamma$-frame matroids for any group $\Gamma$ \cite{Zaslavsky1991}.
Important types of frame matroids include signed-graphic matroids, bicircular matroids, and Dowling geometries.

The lift matroid of a biased graph arises as a special case of a more general construction, which we will now describe.
A collection $\cC$ of circuits of a matroid $M$ is a \emph{linear class} if, whenever $C_1$ and $C_2$ are circuits in $\cC$ so that $|C_1 \cup C_2| - r_M(C_1 \cup C_2) = 2$ (the pair $(C_1, C_2)$ is a \emph{modular pair}), then every circuit $C$ of $M$ contained in $C_1 \cup C_2$ is also in $\cC$.
For example, it is straightforward to check that $(G, \cB)$ is a biased graph if and only if $\cB$ is a linear class of circuits of the graphic matroid $M(G)$.
A classical result, originally due to Crapo \cite{Crapo1965} but stated in the form below by Brylawski \cite{Brylawski1986}, shows how to construct a new matroid on $E(M)$ from a linear class of circuits of $M$.
This new matroid $M'$ is always an \emph{elementary lift} of $M$, which means that there is a matroid $K$ with an element $e$ so that $K \del e = M'$ and $K /e = M$.

\begin{theorem}[{\cite{Brylawski1986, Crapo1965}}]\label{thm: brylawski}
Let $M$ be a matroid on ground set $E$ and let $\cC$ be a linear class of circuits of $M$. 
Then the function $r_{M'} \colon 2^E \to \mathbb Z$ defined, for all $X \subseteq E$, by 
$$r_{M'}(X)=\begin{cases}
r_M(X) & \text{ if each circuit of $M|X$ is in $\cC$},  \\
r_M(X)+1 & \text{ otherwise} 
\end{cases}$$
is the rank function of an elementary lift $M'$ of $M$.
Moreover, every elementary lift of $M$ can be obtained in this way.
\end{theorem}

The \emph{lift matroid} of $(G, \cB)$, denoted $L(G, \cB)$, is the matroid on $E(G)$ obtained by applying Theorem \ref{thm: brylawski} with $M(G)$ and the linear class $\cB$.
The circuits of $L(G, \cB)$ are the balanced cycles, the tight handcuffs with both cycles unbalanced, the theta graphs with all cycles unbalanced, and the pairs of vertex-disjoint unbalanced cycles.
Note that $F(G, \cB) = L(G, \cB)$ if and only if $(G, \cB)$ has no vertex-disjoint unbalanced cycles.
If every cycle is balanced then the circuits of $L(G, \cB)$ are precisely the cycles of $G$, so $L(G, \cB)$ is equal to the graphic matroid of $G$.
For a gain graph $(G, \psi)$ we will write $L(G, \psi)$ for $L(G, \cB_{\psi})$.
A matroid is \emph{lifted-graphic} if it is the lift matroid of a biased graph, and for every group $\Gamma$, a \emph{$\Gamma$-lifted-graphic matroid} is the lift matroid of a biased graph arising from a $\Gamma$-gain graph.
The class of lifted-graphic matroids is minor-closed, as is the class of $\Gamma$-lifted-graphic matroids for any group $\Gamma$ \cite{Zaslavsky1991}.
Important types of lifted-graphic matroids include even-cycle matroids and spikes.

\section{The construction} \label{sec: the construction}

In this section we will prove Theorem \ref{thm: main}, the main result of this paper.
Given a group $\Gamma$ with a normal subgroup $\Gamma_1$ and a partition $\{\Gamma_1\} \cup \cA$ of $\Gamma$ so that each $A \in \cA$ is malnormal and has all conjugates in $\cA$, the following key definition takes in a $\Gamma$-gain graph $(G, \psi)$ and a defines a linear class of circuits of the underlying $(\Gamma/\Gamma_1)$-frame matroid $F(G, \psi/\Gamma_1)$.

\begin{definition} \label{def: the linear class}
Let $\Gamma$ be a group, let $\Gamma_1$ be a normal subgroup of $\Gamma$, and let $\{\Gamma_1\} \cup \cA$ be a partition of $\Gamma$ so that each $A \in \cA$ is malnormal and has all conjugates in $\cA$.
For a $\Gamma$-gain graph $(G, \psi)$, define $\cC(\Gamma, \Gamma_1, \cA, G, \psi)$ to be the set of circuits of the underlying $(\Gamma/\Gamma_1)$-frame matroid $F(G, \psi/\Gamma_1)$ so that $C \in \cC(\Gamma, \Gamma_1, \cA, G, \psi)$ if and only if 
\begin{itemize}
    \item $C$ is a $\psi$-balanced cycle, or
    \item $C$ is a handcuff or a theta graph and there is a switching function $\eta$ and a group $A \in \cA$ so that $\psi^{\eta}(C) \subseteq A$.
\end{itemize}
\end{definition}

Clearly the set $\cC(\Gamma, \Gamma_1, \cA, G, \psi)$ is invariant under switching.
We will prove that $\cC(\Gamma, \Gamma_1, \cA, G, \psi)$ is a linear class of circuits of the frame matroid $F(G, \psi/\Gamma_1)$.
Before this, we will give some alternate characterizations of $\cC(\Gamma, \Gamma_1, \cA, G, \psi)$ which will be useful in the rest of the paper.
Notably, we can also define $\cC(\Gamma, \Gamma_1, \cA, G, \psi)$ using walks on $(G, \psi)$.
Inspired by \cite[Definition 5.8]{Bernstein2022}, we say that a \emph{cyclic covering pair} of walks for a handcuff or theta graph $C$ is a pair $(W_1, W_2)$ of closed walks on $C$, starting at the same vertex, so that each walk contains exactly one cycle of $C$, the walks do not contain the same cycle, and each edge of $C$ is traversed once or twice by the concatenation $W_1W_2$.
Note that there are edges $e_1$ and $e_2$ of $C$ so that $G[C - \{e_1, e_2\}]$ is a tree, and for $i = 1,2$ the edge $e_i$ is in $W_i$ but not $W_{3-i}$.

\begin{lemma} \label{lem: equivalences for the linear class}
Let $\Gamma$ be a group, let $\Gamma_1$ be a normal subgroup of $\Gamma$, and let $\{\Gamma_1\} \cup \cA$ be a partition of $\Gamma$ so that each $A \in \cA$ is malnormal and has all conjugates in $\cA$.
Let $(G, \psi)$ be a $\Gamma$-gain graph and let $C$ be a theta graph or handcuff circuit of $F(G, \psi/\Gamma_1)$.
The following are equivalent:
\begin{enumerate}[label=$(\roman*)$]
\item $C \in \cC(\Gamma, \Gamma_1, \cA, G, \psi)$.

\item For every switching function $\eta$ so that $C$ has a $\psi^{\eta}$-normalized spanning tree, there is a group $A \in \cA$ so that $\psi^{\eta}(C) \subseteq A$.

\item There is a cyclic covering pair $(W_1, W_2)$ of walks for $C$ and a group $A \in \cA$ so that $\psi(W_1), \psi(W_2) \in A$.

\item For every cyclic covering pair $(W_1, W_2)$ of walks for $C$ there is a group $A \in \cA$ so that $\psi(W_1), \psi(W_2) \in A$.
\end{enumerate}
\end{lemma}
\begin{proof}
Let $\circ$ be the group operation of $\Gamma$.
We will show the implications $(i) \Rightarrow (iv) \Rightarrow (iii) \Rightarrow (ii) \Rightarrow (i)$.
Clearly $(iv)$ implies $(iii)$ and $(ii)$ implies $(i)$.
We will first show that $(i)$ implies $(iv)$.
Suppose that $(i)$ holds for switching function $\eta$ and group $A \in \cA$.
Let $(W_1, W_2)$ be any cyclic covering pair of walks for $C$ starting at a vertex $v$.
Since $\psi^{\eta}(C) \subseteq A$ we see that $\psi^{\eta}(W_i) \in A$ for $i = 1,2$.
For $i = 1,2$ we have $\psi(W_i) = \eta(v) \circ \psi^{\eta}(W_i) \circ \eta(v)^{-1}$ (see Section \ref{sec: gain graphs}), so $\psi(W_i) \in \eta(v) \circ A \circ \eta(v)^{-1}$.
Since $\eta(v) \circ A \circ \eta(v)^{-1} \in \cA$ we see that $(iv)$ holds.

We will complete the proof by showing that $(iii)$ implies $(ii)$.
Suppose that $(iii)$ holds for $(W_1, W_2)$ and $A \in \cA$, and let $v$ be the starting vertex of $W_1$ and $W_2$.
Let $\eta$ be a switching function so that $C$ has a $\psi^{\eta}$-normalized spanning tree $T$, and let $e_1,e_2 \in C - T$.
Let $A' = \eta(v)^{-1} \circ A \circ \eta(v)$, so $A' \in \cA$.
Note that $\psi^{\eta}(W_i) = \eta(v)^{-1}\circ \psi(W_i) \circ \eta(v)$ for $i = 1,2$, so $\psi^{\eta}(W_i) \in A'$.
By the definition of cyclic covering pair we see that $W_1$ and $W_2$ each traverse one of $e_1, e_2$, and either $W_1$ or $W_2$ traverses exactly one of $e_1, e_2$.
We may assume that $W_1$ traverses $e_1$ but not $e_2$, and $W_2$ traverses $e_2$.
Since $T$ is $\psi^{\eta}$-normalized this implies that $\psi^{\eta}(W_1) \in \psi^{\eta}(e_1)$, and since $\psi^{\eta}(W_1) \in A'$ we see that $\psi^{\eta}(e_1) \subseteq A'$.
Note that $\psi^{\eta}(W_2)$ is the product of an element in $\psi^{\eta}(e_2)$ and possibly an element in $\psi^{\eta}(e_1)$ (if $W_2$ traverses $e_1$).
Since $\psi^{\eta}(W_2)$ and $\psi^{\eta}(e_1)$ are both in $A'$ it follows that $\psi^{\eta}(e_2)$ is in $A'$.
Since $T$ is $\psi^{\eta}$-normalized and $\psi^{\eta}(e_i)\in A'$ for $i = 1,2$ it follows that $\psi^{\eta}(C) \subseteq A'$, as desired.
\end{proof}

Our proofs in this paper will only make use of the equivalence between $(i)$ and $(ii)$, but the equivalences between $(i)$, $(iii)$, and $(iv)$ will be useful in Section \ref{sec: examples} when we compare Theorem \ref{thm: main} with the constructions of Anderson, Su, and Zaslavsky \cite{Anderson-Su-Zaslavsky2024} and Bernstein \cite{Bernstein2022}.

We will now prove that $\cC(\Gamma, \Gamma_1, \cA, G, \psi)$ is a linear class of circuits of the frame matroid $F(G, \psi/\Gamma_1)$.
Combined with Theorem \ref{thm: brylawski}, this defines an elementary lift of $F(G, \psi/\Gamma_1)$, proving Theorem \ref{thm: main}, our main result.

\begin{theorem} \label{thm: main precise}
Let $\Gamma$ be a group, let $\Gamma_1$ be a normal subgroup of $\Gamma$, let $\{\Gamma_1\} \cup \cA$ be a partition of $\Gamma$ so that each $A \in \cA$ is malnormal and has all conjugates in $\cA$, and let $(G, \psi)$ be a $\Gamma$-gain graph.
Then $\cC(\Gamma, \Gamma_1, \cA, G, \psi)$ is a linear class of circuits of the underlying $(\Gamma/\Gamma_1)$-frame matroid $F(G, \psi/\Gamma_1)$.
\end{theorem}
\begin{proof}
We will write $N$ for the frame matroid $F(G, \psi/\Gamma_1)$, $\cC$ for $\cC(\Gamma, \Gamma_1, \cA, G, \psi)$, and $\ep$ for the identity element of $\Gamma$.
Let $C_1, C_2 \in \cC$ with $|C_1 \cup C_2| = r_N(C_1 \cup C_2) + 2$, and let $C$ be a circuit of $N$ contained in $C_1 \cup C_2$ with $C \notin \{C_1, C_2\}$.
We will show that $C \in \cC$.
We first prove two important properties of the graph $G[C_1 \cup C_2]$.
The \emph{cyclomatic number} of a graph is the minimum number of edges whose deletion leaves a forest.

\begin{claim} \label{claim: cyclomatic number 3}
The graph $G[C_1 \cup C_2]$ is connected and has cyclomatic number at most $3$.
\end{claim}
\begin{proof}
Note that $G[C_1]$ and $G[C_2]$ are both connected graphs.
Then $G[C_1 \cup C_2]$ is connected or else $C \in \{C_1, C_2\}$.
Since $|C_1 \cup C_2| = r_N(C_1 \cup C_2) + 2$, we can delete two edges from $C_1 \cup C_2$ to obtain an independent set $I$ of the frame matroid $N$.
This implies that each component of $G[I]$ has at most one cycle.
If $G[I]$ is connected, then the claim follows.
If $G[I]$ is disconnected, then it has at most three connected components because it is obtained by deleting two edges from a connected graph.
If it has three components, then neither deletion decreased the cyclomatic number and the claim follows.
If it has two components, then one of the two deletions did not decrease the cyclomatic number and the claim follows.
\end{proof}

If $G[C_1 \cup C_2]$ has cyclomatic number $2$ then $C_1$ and $C_2$ are both $\psi$-balanced cycles.
If they share no edges, then $C \in \{C_1, C_2\}$.
If they share an edge then $G[C_1 \cup C_2]$ is a theta graph, and it follows from the theta property for $\psi$-balanced cycles that $C$ is a $\psi$-balanced cycle and is therefore in $\cC$.
So for the remainder of the proof we may assume that $G[C_1 \cup C_2]$ has cyclomatic number $3$.
The next claim deals with $\psi$-balanced cycles.

\begin{claim} \label{claim: no balanced cycles}
If $G[C_1 \cup C_2]$ contains a $\psi$-balanced cycle, then $C \in \cC$.
\end{claim}
\begin{proof}
Suppose $G[C_1 \cup C_2]$ contains a $\psi$-balanced cycle $D$.
Let $T$ be a spanning tree of $G[C_1 \cup C_2]$ that contains all but one edge of $D$.
By switching, we may assume that $T$ is $\psi$-normalized.
Since $D$ is $\psi$-balanced, it follows that $D$ is $\psi$-normalized.
Since $G[C_1 \cup C_2]$ has cyclomatic number $3$, there are two edges $e$ and $f$ in $(C_1 \cup C_2) - (T \cup D)$.
We will first show that there is some $A \in \cA$ that contains $\psi(e)$ and $\psi(f)$.
If $C_1$ is a handcuff or a theta graph then $e,f \in C_1$, and since $C_1 \in \cC$ there is some $A \in \cA$ that contains $\psi(e)$ and $\psi(f)$.
The same reasoning applies if $C_2$ is a handcuff or a theta graph, so we may assume that $C_1$ and $C_2$ are both $\psi$-balanced cycles.
By choosing $D = C_1$ we may assume that $C_1$ is $\psi$-normalized.
Then $e$ and $f$ are both in $C_2$, and since $C_2$ is $\psi$-balanced and every edge other than $e$ and $f$ is $\psi$-normalized it follows that $\psi(e) = \psi(f) = \{\alpha, \alpha^{-1}\}$ for some $\alpha \in \Gamma - \{\ep\}$.
If $\alpha \in \Gamma_1$ then $C_1 \cup C_2$ is $(\psi/\Gamma_1)$-normalized and so $r_N(C_1 \cup C_2) = |V(G[C_1 \cup C_2])| - 1 = |C_1 \cup C_2| - 2 - 1$ because $G[C_1 \cup C_2]$ has cyclomatic number $3$.
But then $|C_1 \cup C_2| - r_N(C_1 \cup C_2) = 3$, a contradiction.
So $\alpha \notin \Gamma_1$, so there is some $A \in \cA$ with $\alpha \in A$, and therefore $A$ contains $\psi(e)$ and $\psi(f)$.

If $C$ is a handcuff or a theta graph, then $e,f \in C$ and it follows that $C \in \cC$.
So we may assume that $C$ is a $(\psi/\Gamma_1)$-balanced cycle.
We must show that $C$ is $\psi$-balanced.
If $C$ contains neither $e$ nor $f$, then $C$ is $\psi$-normalized.
If $C$ contains exactly one of $e$ and $f$, then $C$ is not $(\psi/\Gamma_1)$-balanced because $\alpha \notin \Gamma_1$, a contradiction.
So $C$ contains $e$ and $f$.
Let $W$ be a simple closed walk around $C$.
Then $\psi(W)$ is the product of an element in $\psi(e)$ and an element in $\psi(f)$, so $\psi(W) \in A$.
And $\psi(W) \in \Gamma_1$ because $C$ is $(\psi/\Gamma_1)$-balanced.
So $\psi(W) = \ep$ because $A \cap \Gamma_1 = \{\ep\}$, and therefore $C$ is $\psi$-balanced and in $\cC$, as desired.
\end{proof}

By Claim \ref{claim: no balanced cycles} we may assume that $G[C_1 \cup C_2]$ contains no $\psi$-balanced cycles. 
This implies that $C_1$ and $C_2$ are not $(\psi/\Gamma_1)$-balanced; if $C_i$ were $(\psi/\Gamma_1)$-balanced, it would be $\psi$-balanced because it is in $\cC$.
Note that $C$ may be $(\psi/\Gamma_1)$-balanced.
We consider one more special case.

\begin{claim} \label{claim: only one cycle}
If $G[C_1 \cap C_2]$ has at most one cycle, then $C \in \cC$.
\end{claim}
\begin{proof}
Suppose that $G[C_1 \cap C_2]$ has at most one cycle.
Then there is an edge $e_1 \in C_1 - C_2$ that is in a cycle of $G[C_1]$ and an edge $e_2 \in C_2 - C_1$ that is in a cycle of $G[C_2]$.
Moreover, $G[C_1] \del e_1$ and $G[C_2] \del e_2$ are both connected, and since they share at least one vertex it follows that $G[C_1 \cup C_2] \del \{e_1, e_2\}$ is connected.
Since $G[C_1 \cup C_2]$ has cyclomatic number $3$ we see that $G[C_1 \cup C_2] \del \{e_1, e_2\}$ has exactly one cycle.
Let $e_3$ be an edge of this cycle, and note that $G[C_1 \cup C_2] \del \{e_1, e_2, e_3\}$ is a spanning tree of $G[C_1 \cup C_2]$.
By switching we may assume that it is $\psi$-normalized.
Note that for $i = 1,2,3$, $\psi(e_i) \ne \{\ep\}$ because $G[C_1 \cup C_2]$ contains no $\psi$-balanced cycles.
Since $C_1$ contains $e_1$ and $e_3$ but not $e_2$ and $C_1$ is a handcuff or theta graph, there is some $A_1 \in \cA$ that contains $\psi(e_1)$ and $\psi(e_3)$ since $C_1 \in \cC$.
Similarly, there is some $A_2 \in \cA$ that contains $\psi(e_2)$ and $\psi(e_3)$.
Then $A_1 = A_2$ because both contain $\psi(e_3)$.
Let $T$ be a spanning tree of $G[C]$.
By switching within $A_1$ we can normalize $T$ so that the edge(s) in $C - T$ have gain values in $A_1$.
If $C$ is a theta graph or a handcuff this implies that $C \in \cC$, so we may assume that $C$ is a $(\psi/\Gamma_1)$-balanced cycle.
Let $W$ be a simple closed walk around $C$.
Clearly $\psi(W) \in A_1$.
Also, $\psi(W) \in \Gamma_1$ since $C$ is $(\psi/\Gamma_1)$-balanced.
Therefore $\psi(W) = \ep$ since $A_1 \cap \Gamma_1 = \{\ep\}$, so $C \in \cC$.
\end{proof}

By Claim \ref{claim: only one cycle} we may assume that $G[C_1 \cap C_2]$ contains at least two cycles.
And by Claim \ref{claim: no balanced cycles} we may assume that $C_1$ and $C_2$ are not cycles.
Since $C_1$ has cyclomatic number two it follows that $G[C_1 \cap C_2]$ has cyclomatic number two.
Therefore for $i = 1,2$, $C_i$ is a non-minimal graph with cyclomatic number two, so $C_1$ and $C_2$ are loose handcuffs that contain the same pair of cycles, say $D$ and $D'$.
For $i = 1,2$, let $P_i$ be the path in $G[C_i]$ between $V(D)$ and $V(D')$.
Let $v$ be the end of $P_2$ in $V(D)$.
Let $e$ and $f$ be edges of $D$ and $P_2$, respectively, that are incident with $v$.
Let $e'$ be any edge of $D'$.
By switching we may assume that the spanning tree $(C_1 \cup C_2) - \{e, e', f\}$ of $G[C_1 \cup C_2]$ is $\psi$-normalized.
Since $C_1$ contains $e$ and $e'$ but not $f$ and $C_1 \in \cC$, there is some $A_1 \in \cA$ that contains $\psi(e)$ and $\psi(e')$.
Let $\{w,v\}$ and $\{u,v\}$ be the sets of ends of $e$ and $f$, respectively.
Let $\psi'$ be the $\Gamma$-gain function obtained from $\psi$ by switching by $\psi(f, u, v)^{-1}$ at each vertex in $V(D)$.
Then $\psi'(f) = \ep$ and $\psi'(e') = \psi(e')$, while $\psi'(e, v, w) = \psi(f, u, v) \circ \psi(e, v, w) \circ \psi(f, u, v)^{-1}$ and $\psi'(g) = \{\ep\}$ for each edge $g$ in $C_2 - \{e, e'\}$.
Since $C_2 \in \cC$, there is some $A_2 \in \cA$ that contains $\psi'(e')$ and $\psi'(e, v, w)$.
Since $\psi'(e') = \psi(e')$ we see that $A_2 = A_1$.
Since $A_1$ contains $\psi(e, v, w)$ and $\psi(f, u, v) \circ \psi(e, v, w) \circ \psi(f, u, v)^{-1}$ and $A_1$ is malnormal it follows that $\psi(f, u, v) \in A_1$.
So $A_1$ contains $\psi(e) \cup \psi(e') \cup \psi(f)$, and therefore $A_1$ contains $\psi(C_1 \cup C_2)$.
Let $T$ be a spanning tree of $G[C]$.
By starting with $\psi$ and switching within $A_1$ we can $\psi$-normalize $T$ so that the edge(s) in $T - C$ have gain values in $A_1$.
If $C$ is a theta graph or a handcuff this implies that $C \in \cC$, so we may assume that $C$ is a $(\psi/\Gamma_1)$-balanced cycle.
Let $W$ be a simple closed walk around $C$.
Clearly $\psi(W) \in A_1$.
Also, $\psi(W) \in \Gamma_1$ since $C$ is $(\psi/\Gamma_1)$-balanced.
Therefore $\psi(W) = \ep$ since $A_1 \cap \Gamma_1 = \{\ep\}$, so $C \in \cC$.
\end{proof}

Theorem \ref{thm: main precise} allows us to define, via Theorem \ref{thm: brylawski}, the elementary lift of $F(G, \psi/\Gamma_1)$ associated with the linear class $\cC(\Gamma, \Gamma_1, \cA, G, \psi)$; this is the matroid of Theorem \ref{thm: main}.

\begin{definition} \label{def: the matroid}
Let $\Gamma$ be a group, let $\Gamma_1$ be a normal subgroup of $\Gamma$, let $\{\Gamma_1\} \cup \cA$ be a partition of $\Gamma$ so that each $A \in \cA$ is malnormal and has all conjugates in $\cA$, and let $(G, \psi)$ be a $\Gamma$-gain graph.
Define $M(\Gamma, \Gamma_1, \cA, G, \psi)$ to be the elementary lift of $F(G, \psi/\Gamma_1)$ given by Theorem \ref{thm: brylawski} for the linear class $\cC(\Gamma, \Gamma_1, \cA, G, \psi)$.
\end{definition}

The following properties follow directly from Definitions \ref{def: the linear class} and \ref{def: the matroid} and will be useful throughout the rest of the paper.

\begin{lemma} \label{lem: placement of loops}
Let $\Gamma$ be a group with identity $\ep$, let $\Gamma_1$ be a normal subgroup of $\Gamma$, and let $\{\Gamma_1\} \cup \cA$ be a partition of $\Gamma$ so that each $A \in \cA$ is malnormal and has all conjugates in $\cA$.
Let $(G, \psi)$ be a $\Gamma$-gain graph and let $e$ be a loop of $G$ at vertex $v$. 
\begin{enumerate}[label=$(\alph*)$]
    \item If $\psi(e, v, v) \in \Gamma_1 - \{\ep\}$, then replacing $\psi(e, v, v)$ with any other element in $\Gamma_1 - \{\ep\}$ does not change the associated matroid.

    \item If $\psi(e, v, v) \in A - \{\ep\}$ for some $A \in \cA$, then replacing $\psi(e, v, v)$ with any other element in $A - \{\ep\}$ does not change the associated matroid.

    \item If $\psi(e, v, v) \in \Gamma_1 - \{\ep\}$ and $(G', \psi')$ is the gain graph obtained from $(G, \psi)$ by moving $e$ to a different vertex $w$ and setting $\psi'(e, w, w) = \psi(e, v, v)$, then $M(\Gamma, \Gamma_1, \cA, G', \psi') = M(\Gamma, \Gamma_1, \cA, G, \psi)$.
\end{enumerate}
\end{lemma}

In Section \ref{sec: properties} we will prove some important properties of $M(\Gamma, \Gamma_1, \cA, G, \psi)$, and then in Section \ref{sec: examples} we will look at examples of $M(\Gamma, \Gamma_1, \cA, G, \psi)$ for interesting choices of $\Gamma$.

We will conclude this section with a comment about the proof of Theorem \ref{thm: main precise}.
By examining the proof we see that the groups in $\cA$ are only required to be malnormal in the very last paragraph.
In this paragraph, there is a pair of vertex-disjoint $(\psi/\Gamma_1)$-unbalanced cycles.
So if we restrict our attention to $\Gamma$-gain graphs with no vertex-disjoint $(\psi/\Gamma_1)$-unbalanced cycles then we can obtain a more general construction in which the groups in $\cA$ are not required to be malnormal.
Since biased graphs with no vertex-disjoint unbalanced cycles are well-studied and complex (see \cite{Chen-Pivotte-2018}), this may be an interesting direction for future work.

\section{Properties} \label{sec: properties}

In this section we will fix a group $\Gamma$ with identity $\ep$, a normal subgroup $\Gamma_1$ of $\Gamma$, and a partition $\{\Gamma_1\} \cup \cA$ of $\Gamma$ so that each $A \in \cA$ is malnormal and has all conjugates in $\cA$.
We will study properties of the following class of matroids.

\begin{definition} \label{def: the class of matroids}
Let $\cM(\Gamma, \Gamma_1, \cA)$ be the class of matroids so that $M \in \cM(\Gamma, \Gamma_1, \cA)$ if there is a $\Gamma$-gain graph $(G, \psi)$ so that $M \cong M(\Gamma, \Gamma_1, \cA, G, \psi)$.
\end{definition}

We will write down the rank function, bases, and circuits of a given matroid in $\cM(\Gamma, \Gamma_1, \cA)$, prove that $\cM(\Gamma, \Gamma_1, \cA)$ is minor-closed if $\Gamma/\Gamma_1$ is isomorphic to a group in $\cA$ (which is always the case when $\Gamma$ is finite by Proposition \ref{prop: Frobenius properties}), show that $\cM(\Gamma, \Gamma_1, \cA)$ is not closed under duality, and show that $\cM(\Gamma, \Gamma_1, \cA)$ is not closed under direct sums unless $\Gamma_1$ is trivial.

\subsection{Rank, bases, and circuits} \label{sec: rank, bases, circuits}

The following proposition gives the rank function of a matroid in $\cM(\Gamma, \Gamma_1, \cA)$.
The proof, which we omit, follows directly from Proposition \ref{prop: rank of X in frame matroid} together with the application of Theorem \ref{thm: brylawski} to the linear class of Definition \ref{def: the linear class}.

\begin{proposition} \label{prop: rank function}
Let $(G, \psi)$ be a $\Gamma$-gain graph, let $M = M(\Gamma, \Gamma_1, \cA, G, \psi)$, and let $N$ be the frame matroid of $(G, \psi/\Gamma_1)$.
Then for every set $X$ of edges of $G$, the rank of $X$ in the matroid $M(\Gamma, \Gamma_1, \cA, G, \psi)$ is
$$r(X) = |V(G[X])| - b(X) + l(X),$$
where $b(X)$ is the number of $(\psi/\Gamma_1)$-balanced components of $G[X]$, and $l(X) = 0$ if every circuit of $N|X$ is in $\cC(\Gamma, \Gamma_1, \cA, G, \psi)$ and $l(X) = 1$ otherwise.
\end{proposition}

By comparing with the rank functions for lifted-graphic matroids and frame matroids (see \cite[page 16]{Bowler-Funk-Slilaty-2020}, for example), we obtain the following corollary.

\begin{proposition} \label{prop: special cases frame and lift matroids}
Let $(G, \psi)$ be a $\Gamma$-gain graph and let $M = M(\Gamma, \Gamma_1, \cA, G, \psi)$.
\begin{enumerate}[label=$(\alph*)$]
\item If $\Gamma_1$ is trivial then $M$ is the frame matroid of $(G, \psi)$.

\item If $\Gamma_1 = \Gamma$ then $M$ is the lift matroid of $(G, \psi)$.
\end{enumerate}
\end{proposition}

By choosing a $\Gamma$-gain function with image contained in $\Gamma_1$ or in some $A \in \cA$ we obtain the following corollary.

\begin{proposition} \label{prop: special cases}
$\cM(\Gamma, \Gamma_1, \cA)$ contains the class of $\Gamma_1$-lifted-graphic matroids, and for each $A \in \cA$, $\cM(\Gamma, \Gamma_1, \cA)$ contains the class of $A$-frame matroids.
\end{proposition}

From Proposition \ref{prop: rank function} we can also deduce the sets of bases, circuits, hyperplanes, independent sets, flats, etc., of a matroid in $\cM(\Gamma, \Gamma_1, \cA)$.
We will only explicitly state the sets of bases and circuits as they are most likely to be useful for future work.

\begin{proposition} \label{prop: bases and circuits}
Let $(G, \psi)$ be a $\Gamma$-gain graph, let $M = M(\Gamma, \Gamma_1, \cA, G, \psi)$, and let $N$ be the frame matroid of $(G, \psi/\Gamma_1)$.
\begin{enumerate}[label=$(\alph*)$]
    \item $B \subseteq E(G)$ is a basis of $M$ if and only if $B$ spans $N$, contains a unique circuit of $N$, and this circuit is not in $\cC(\Gamma, \Gamma_1, \cA, G, \psi)$.

    \item $C \subseteq E(G)$ is a circuit of $M$ if and only if $C \in \cC(\Gamma, \Gamma_1, \cA, G, \psi)$ or $C$ contains no sets in $\cC$ and is minimal with the property that $|C| - r_N(C) = 2$.
\end{enumerate}
\end{proposition}

\subsection{Minors} \label{sec: minors}

We will next show that $\cM(\Gamma, \Gamma_1, \cA)$ is closed under the minor operations of deletion and contraction when $\Gamma/\Gamma_1$ is trivial or isomorphic to a group in $\cA$.
We will consider four cases, and we will only need this assumption about $\Gamma/\Gamma_1$ in one case (Lemma \ref{lem: loop contraction easy case}).
Throughout Section \ref{sec: minors} we will fix a $\Gamma$-gain graph $(G, \psi)$, the underlying frame matroid $N = F(G, \psi/\Gamma_1)$, the linear class $\cC = \cC(\Gamma, \Gamma_1, \cA, G, \psi)$ of circuits of $N$, the matroid $M = M(\Gamma, \Gamma_1, \cA, G, \psi)$, and an element $e \in E(G)$.
For a set $X \subseteq E(G)$ we will write $\psi|_{X}$ for the restriction of $\psi$ to the set of orientations of edges in $X$.

We omit the straightforward proof of the following lemma, which shows that $\cM(\Gamma, \Gamma_1, \cA)$ is closed under deletion.

\begin{lemma} \label{lem: deletion}
$M \del e \cong M(\Gamma, \Gamma_1, \cA, G\del e, \psi|_{E(G) - e})$.
\end{lemma}

We next consider the contraction of $e$.
The following lemma considers the case in which $e$ is a non-loop of $G$.
By switching we may assume that $e$ is $\psi$-normalized.

\begin{lemma} \label{lem: contraction of a non-loop}
If $e$ is a $\psi$-normalized non-loop of $G$, then $M / e \cong M(\Gamma, \Gamma_1, \cA, G/ e, \psi|_{E(G) - e})$.
\end{lemma}
\begin{proof}
Let $\psi' = \psi|_{E(G) - e}$, let $\cC' = \cC(\Gamma, \Gamma_1, \cA, G/e, \psi')$, and let $M' = M(\Gamma, \Gamma_1, \cA, G/e, \psi')$.
It is well-known that the underlying frame matroid $F(G/e, \psi'/\Gamma_1)$ is isomorphic to $N/e$ (see \cite[Proposition 2.4]{Funk-Pivotto-Slilaty-2022}, for example), so $\cC'$ is a linear class of circuits of $N/e$.
For a set $X \subseteq E(G) - e$, let $b'(X)$ be the number of $(\psi'/\Gamma_1)$-balanced components of $G/e[X]$, and let $l'(X) = 0$ if every circuit of $(N/e)|X$ is in $\cC'$ and $l'(X) = 1$ otherwise.
By Proposition \ref{prop: rank function} applied to $(G/e, \psi')$, every set $X \subseteq E(G) - e$ satisfies
$$r_{M'}(X) = |V(G/e[X])| - b'(X) + l'(X).$$
We will show that every set $X \subseteq E(G) - e$ satisfies $r_{M/e}(X) = r_{M'}(X)$.
If $e$ has at most one end in $V(G[X])$, then the gain graphs $(G[X], \psi|_X)$ and $(G/e[X], \psi'|_X)$ are equal and it follows that $r_{M'}(X) = r_M(X)$.
From Proposition \ref{prop: rank function} we see that $r_M(X \cup e) = r_M(X) + 1$, so $r_{M/e}(X) = r_M(X)$ and therefore $r_{M'}(X) = r_{M/e}(X)$, as desired.
So we may assume that $e$ has both ends in $V(G[X])$, which implies that $|V(G/e[X])| = |V(G[X])| - 1$.
We first prove two claims to compare $G/e[X]$ with $G[X \cup e]$.

\begin{claim} \label{claim: b-value comparison in non-loop case}
$b'(X) = b(X \cup e)$.
\end{claim}
\begin{proof}
By Proposition \ref{prop: rank of X in frame matroid} we see that $r_N(X \cup e) = |V(G[X \cup e])| - b(X \cup e)$, and since $N/e \cong F(G', \psi'/\Gamma_1)$ we see that $r_{N/e}(X) = |V(G'[X])| - b'(X)$.
Since $r_{N/e}(X) = r_N(X \cup e) - 1$ and $|V(G'[X])| = |V(G[X \cup e])| - 1$ we see that $b'(X) = b(X \cup e)$.
\end{proof}

\begin{claim} \label{claim: l-value comparison in non-loop case}
$l'(X) = l(X \cup e)$.
\end{claim}
\begin{proof}
First suppose that $l'(X) = 1$, and let $C'$ be a circuit of $(N/e)|X$ that is not in $\cC'$.
Then $C'$ or $C' \cup e$ is a circuit of $N$, and since $e$ is $\psi$-normalized it follows that this circuit of $N$ is not in $\cC$, so $l(X \cup e) = 1$.
Conversely, suppose that $l(X \cup e) = 1$.
Let $C$ be a circuit of $N|(X \cup e)$ that is not in $\cC$.
We consider three cases: $e \in C$, $e \notin \cl_N(C)$, or $e \in \cl_N(C) - C$.
If $e \in C$, then $C - e$ is a circuit of $N/e$, and since $e$ is $\psi$-normalized it follows that $C - e$ is not in $\cC'$, so $l'(X) = 1$.   
If $e \notin \cl_N(C)$, then $C$ is a circuit of $N/e$, and $C \notin \cC'$ because $e$ is $\psi$-normalized.
So $e \in \cl_N(C) - C$.
Since $e$ is not a loop of $N$ there are circuits $C_1$ and $C_2$ of $N$ contained in $C \cup e$ that contain $e$.
Since $\cC$ is a linear class of circuits of $N$ and $C \notin \cC$, there is some $i \in \{1,2\}$ so that $C_i \notin \cC$.
Then $e \in C_i$, and by previous arguments $C_i - e$ is a circuit of $N/e$ that is not in $\cC'$, so $l'(X) = 1$.
\end{proof}

Using Claims \ref{claim: b-value comparison in non-loop case} and \ref{claim: l-value comparison in non-loop case} we have
\begin{align}
    r_{M'}(X) &= |V(G/e[X])| - b'(X) + l'(X) \\
    &= |V(G[X])| - 1 - b(X \cup e) + l(X \cup e).
\end{align}
If $r_{M/e}(X) = r_M(X)$, then $r_M(X \cup e) = r_M(X) + 1$, and it follows from Proposition \ref{prop: rank function} that exactly one of the following holds: $l(X \cup e) = l(X) + 1$ or $b(X \cup e) = b(X) - 1$.
In either case we see from (2) that $r_{M'}(X) = r_M(X)$.
If $r_{M/e}(X) = r_M(X) - 1$, then $r_M(X \cup e) = r_M(X)$. 
It follows from Proposition \ref{prop: rank function} that $b(X \cup e) = b(X)$ and $l(X \cup e) = l(X)$, and we see from (2) that $r_{M'}(X) = r_M(X) - 1$, as desired.
\end{proof}

We next consider the case in which $e$ is a loop at vertex $v$ so that $\psi(e, v, v) \notin \Gamma_1 - \{\ep\}$.
This is the most complex case, in part because contraction of unbalanced loops for frame matroids is complex.
If $\psi(e,v,v) = \ep$ then $e$ is a loop of $M$ and therefore $M/e = M\del e$, so we may assume that $\psi(e,v,v) \ne \ep$.
Then $\psi(e, v, v)$ is in a unique set $A \in \cA$.
Let $G'$ be the graph with vertex set $V(G)$ and edge set $E(G) - e$ so that $G'$ is obtained from $G$ by replacing each non-loop edge $f$ with ends $v$ and $w$ with a loop at $w$.
Then let $\psi'$ be the following $\Gamma$-gain function for $G'$:
\begin{itemize}
    \item If $f$ is an edge of $G$ with ends $u$ and $w$ with $v \notin \{u,w\}$, define $\psi'(f, u, w) = \psi(f, u, w)$.

    \item If $f$ is a non-loop of $G$ with ends $v$ and $w$, define $\psi'(f, w, w) = \psi(f, v, w)^{-1} \circ \psi(e, v, v) \circ \psi(f, v, w)$.
    
    \item If $f$ is a loop of $G$ at $v$ with $\psi(f, v, v) \in A$, define $\psi'(f, v, v) = \ep$.

    \item If $f$ is a loop of $G$ at $v$ with $\psi(f, v, v) \notin A$, define $\psi'(f, v, v)$ to be any element in $\Gamma_1 - \{\ep\}$.
\end{itemize}
Recall from Lemma \ref{lem: placement of loops}$(a)$ that $M(\Gamma, \Gamma_1, \cA, G', \psi')$ is independent of the choice of element in $\Gamma_1 - \{\ep\}$, so $\psi'$ is well-defined.
We next show that the matroid arising from $(G', \psi')$ is isomorphic to the contraction of $e$ from $M$.

\begin{lemma} \label{lem: contraction of doubly-unbalanced loop}
If $e$ is a loop of $G$ with $\psi(e) \cap \Gamma_1 = \varnothing$, then $M/e \cong M(\Gamma, \Gamma_1, \cA, G', \psi')$.
\end{lemma}
\begin{proof}
Let $\cC' = \cC(\Gamma, \Gamma_1, \cA, G', \psi')$ and let $M' = M(\Gamma, \Gamma_1, \cA, G', \psi')$.
It is well-known that the underlying frame matroid $F(G', \psi'/\Gamma_1)$ is isomorphic to $N/e$ (see \cite[Proposition 2.4]{Funk-Pivotto-Slilaty-2022}, for example), so $\cC'$ is a linear class of circuits of $N/e$.
For a set $X \subseteq E(G) - e$, let $b'(X)$ be the number of $(\psi'/\Gamma_1)$-balanced components of $G'[X]$, and let $l'(X) = 0$ if every circuit of $(N/e)|X$ is in $\cC'$ and $l'(X) = 1$ otherwise.
By Proposition \ref{prop: rank function}, each set $X \subseteq E(G) - e$ satisfies
$$r_{M'}(X) = |V(G'[X])| - b'(X) + l'(X).$$
We will show that every set $X \subseteq E(G) - e$ satisfies $r_{M/e}(X) = r_{M'}(X)$.
Let $Y \subseteq X$ be the set of loops of $G[X]$ at $v$; this set may be empty.
If $v \notin V(G[X])$, then $G'[X] = G[X]$ and it follows that $r_{M'}(X) = r_M(X)$.
From Proposition \ref{prop: rank function} we see that $r_M(X \cup e) = r_M(X) + 1$, so $r_{M/e}(X) = r_M(X)$ and therefore $r_{M'}(X) = r_{M/e}(X)$, as desired.
So we may assume that $v \in V(G[X])$.
We first prove two claims to compare $G'[X]$ with $G[X \cup e]$.

\begin{claim} \label{claim: b-value comparison in loop case}
$|V(G'[X])| - b'(X) = |V(G[X \cup e])| - b(X \cup e) - 1$.
\end{claim}
\begin{proof}
By Proposition \ref{prop: rank of X in frame matroid} we see that $r_N(X \cup e) = |V(G[X \cup e])| - b(X \cup e)$, and since $N/e \cong F(G', \psi'/\Gamma_1)$ we see that $r_{N/e}(X) = |V(G'[X])| - b'(X)$.
Since $r_{N/e}(X) = r_N(X \cup e) - 1$, the statement holds.
\end{proof}

\begin{claim} \label{claim: l-value comparison in loop case}
$l'(X) = l(X \cup e)$.
\end{claim}
\begin{proof}
First suppose that $l(X \cup e) = 0$.
We will show that $l'(X) = 0$.
If $H$ is a component of $G[X]$ that does not contain $v$, then $H$ is a component of $G'[X]$ and $\psi$ and $\psi'$ agree on $E(H)$, so all circuits of $(N/e)|E(H)$ are in $\cC'$.
Since each circuit of $(N/e)|X$ is contained in a component of $G'[X]$, we may assume that $G[X]$ is connected.
Let $T$ be a spanning tree of $G[X]$.
By switching we may assume that $T$ is $\psi$-normalized.
Recall that $\psi(e, v, v) \in A$.
We claim that $\psi(X \cup e) \subseteq A$.
Let $f \in X - T$.
Then $T \cup \{e,f\}$ contains a unique circuit $C$ of $N$.
If $C$ is a $(\psi/\Gamma_1)$-balanced cycle then $C$ is $\psi$-balanced because $l(X \cup e) = 0$, and therefore $f$ is $\psi$-normalized.
If $C$ is not a $(\psi/\Gamma_1)$-balanced cycle, then $C$ is a handcuff in $\cC$ with a $\psi$-normalized spanning tree $F$ with $C - F = \{e,f\}$.
Since $\psi(e,v,v) \in A$ and $C \in \cC$ it follows that $\psi(f) \subseteq A$.
Therefore $\psi(X \cup e) \subseteq A$.
From the definition of $\psi'$ we see that $\psi'(X) \subseteq A$, so every handcuff or theta circuit of $G'[X]$ is in $\cC'$.
Let $C'$ be a $(\psi'/\Gamma_1)$-balanced cycle of $G'[X]$.
If $C'$ is not a loop then $C'$ is a cycle of $G[X]$ with $v \notin V(C)$ and since $\psi$ and $\psi'$ agree on $E(G - v)$ and $l(X \cup e) = 0$ it follows that $C'$ is $\psi'$-balanced so $C' \in \cC$.
So $C'$ is a loop.
Since $l(X \cup e) = 0$ and $\psi(e, v, v) \in A$, every loop $f$ of $G[X]$ at $v$ satisfies $\psi(f, v, v) \in A$.
From the definition of $\psi'$ we see that $f$ is $\psi'$-normalized.
We have shown that every circuit of $(N/e)|X$ is in $\cC'$, so $l'(X) = 0$.

Conversely, suppose that $l(X \cup e) = 1$.
Let $C$ be a circuit of $N|(X \cup e)$ that is not in $\cC$.
We consider three cases: $e \notin \cl_N(C)$, $e \in C$, or $e \in \cl_N(C) - C$.
First suppose that $e \notin \cl_N(C)$.
Then $C$ is a circuit of $N/e$.
If $v \notin V(G[C])$ then $G'[C] = G[C]$ and so $C \notin \cC'$.
So $v \in V(G[C])$.
Since $e \notin \cl_N(C)$ this implies that $C$ is a $(\psi/\Gamma_1)$-balanced cycle.
If $C$ is a loop $f$ at $v$ in $G$, then $C$ is loop $f$ at $v$ in $G'$.
Since $\psi(f, v, v) \ne \ep$ it follows that $\psi'(f, v, v) \ne \ep$ and therefore $C \notin \cC'$, as desired.
So we may assume that $C$ is not a loop at $v$, and is therefore not a loop.
Let $f_1,f_2 \in C$ be the two edges incident with $v$, and for $i = 1,2$ let $w_i$ be the end of $f_i$ other than $v$.
By switching we may assume that $C - f_1$ is $\psi$-normalized, which implies that $\psi(f_1, v, w_1) \in \Gamma_1 - \{\ep\}$.
Then $G'[C]$ is a $\psi$-normalized handcuff with loops $f_1$ and $f_2$ with $\psi'(f_1, w_1, w_1) = \psi(f_1, v, w_1)^{-1} \circ \psi(e, v, v) \circ \psi(f_1, v, w_1)$ and $\psi'(f_2, w_2, w_2) = \psi(e, v, v)$.
Since $A$ is malnormal and $\psi(f_1, v, w_1) \notin A$ we see that $\psi(f_1, v, w_1)^{-1} \circ \psi(e, v, v) \circ \psi(f_1, v, w_1)$ is not in $A$.
Therefore $C' \notin \cC'$ because $\psi(e, v, v) \in A$, and so $l'(X) = 1$, as desired.

Next suppose that $e \in C$.
Then $C - e$ is a circuit of $N/e$, and we will show that $C - e \notin \cC'$, which will imply that $l'(X) = 1$.
Let $T$ be a spanning tree of $G[C]$ and let $f \in C - (T \cup e)$.
By switching we may assume that $T$ is $\psi$-normalized.
Since $C$ is not in $\cC$ we see that $\psi(f)$ is not contained in $A$.
First suppose that $C$ is a loose handcuff, and let $g$ be the non-loop edge incident with $v$.
Then $G'[C - e]$ is a $\psi'$-normalized handcuff, $\psi'(g) = \psi(e)$, and $\psi'(f) = \psi(f)$.
Since $\psi(f)$ is not contained in $A$ it follows that $C - e$ is not in $\cC'$, so $l'(X) = 1$.
Next suppose that $C$ is a tight handcuff.
Let $D$ be the cycle of $G[C]$ other than $e$.
If $D$ is a loop, then in $G'$ we see that $D$ is a $(\psi'/\Gamma_1)$-balanced cycle that is not $\psi'$-balanced, so $l'(X) = 1$.
So we may assume that $D$ is not a loop.
We may assume that $f$ has ends $v$ and $w$ with $v \ne w$.
Let $g$ be the other edge of $G[D]$ with $v$ as an end.
Then $G'[C - e]$ is a $\psi'$-normalized handcuff,
$\psi'(g) = \psi(e)$, and $\psi'(f)$ contains $\psi(f, v, w)^{-1} \circ \psi(e, v, v) \circ \psi(f, v, w)$ and its inverse.
Since $\psi(f, v, w) \notin A$ and $A$ is malnormal it follows that $C - e$ is not in $\cC'$, so $l'(X) = 1$.

Finally, suppose that $e \in \cl_N(C) - C$.
Since $e$ is not a loop of $N$ there are circuits $C_1$ and $C_2$ of $N$ contained in $C \cup e$ that contain $e$.
Since $\cC$ is a linear class of circuits of $N$ and $C \notin \cC$, there is some $i \in \{1,2\}$ so that $C_i \notin \cC$.
Then $e \in C_i$, and by previous arguments $C_i - e$ is a circuit of $N/e$ that is not in $\cC'$, so $l'(X) = 1$.
\end{proof}

By Claims \ref{claim: b-value comparison in loop case} and \ref{claim: l-value comparison in loop case} we have
\begin{align}
    r_{M'}(X) &= |V(G'[X])| - b'(X) + l'(X) \\
    &= |V(G[X \cup e])| - b(X \cup e) - 1 + l'(X) \\
    &= |V(G[X \cup e])| - b(X \cup e) - 1 + l(X \cup e) \\
    &= |V(G[X])| - b(X \cup e) - 1 + l(X \cup e).
\end{align}
If $r_{M/e}(X) = r_M(X)$, then $r_M(X \cup e) = r_M(X) + 1$, and it follows from Proposition \ref{prop: rank function} that exactly one of the following holds: $l(X \cup e) = l(X) + 1$ or $b(G[X \cup e]) = b(G[X]) - 1$.
In either case we see from (6) that $r_{M'}(X) = r_M(X)$.
If $r_{M/e}(X) = r_M(X) - 1$, then $r_M(X \cup e) = r_M(X)$.
It then follows from Proposition \ref{prop: rank function} that $V(G[X \cup e]) = V(G[X])$, while $b(X \cup e) = b(X)$ and $l(X \cup e) = l(X)$, and we see from (6) that $r_{M'}(X) = r_M(X) - 1$, as desired.
\end{proof}

Finally, we will consider the contraction of $e$ from $M$ in the case that $e$ is a loop at vertex $v$ with $\psi(e, v, v) \in \Gamma_1 - \{\ep\}$.
In this case we will need to assume that $\Gamma/\Gamma_1$ is isomorphic to a group in $\cA$.
With this assumption we can simply delete $e$ and replace $\psi$ with $\psi/\Gamma_1$.

\begin{lemma} \label{lem: loop contraction easy case}
If $e$ is a loop of $G$ with $\psi(e) \subseteq \Gamma_1 - \{\ep\}$ and $\Gamma/\Gamma_1$ is trivial or isomorphic to a group in $\cA$, then $M/e \cong M(\Gamma, \Gamma_1, \cA, G \del e, (\psi/\Gamma_1)|_{E(G) - e})$.
\end{lemma}
\begin{proof}
By \cite[pg. 282, Exercise 11]{Oxley2011} we know that $M/e$ is an elementary lift of $N/e$.
Since $e$ is a loop of $N$ but not $M$ we have $r(M/e) = r(N/e)$, which implies that $M/e = N/e$ \cite[Corollary 7.3.4]{Oxley2011}.
And since $e$ is a loop of $N$ we see that $N/e = N\del e$, so $N/e$ is isomorphic to the frame matroid $F(G \del e, (\psi/\Gamma_1)|_{E(G) - e})$.
If $\Gamma/\Gamma_1$ is nontrivial, let $A \in \cA$ so that $A \cong \Gamma/\Gamma_1$.
We see from Definition \ref{def: the matroid} that every circuit of $N \del e$ is in $\cC(\Gamma, \Gamma_1, \cA, G \del e, (\psi/\Gamma_1)|_{E(G) - e})$.
If $\Gamma/\Gamma_1$ is nontrivial, this is because we can view $(G \del e, (\psi/\Gamma_1)|_{E(G) - e})$ as an $A$-gain graph, and if $\Gamma/\Gamma_1$ is trivial, this is because we can view $(G \del e, (\psi/\Gamma_1)|_{E(G) - e})$ as an $\{\ep\}$-gain graph.
Therefore $N \del e \cong M(\Gamma, \Gamma_1, \cA, G \del e, (\psi/\Gamma_1)|_{E(G) - e})$, as desired.
\end{proof}

Combining the previous four lemmas gives the following.

\begin{theorem} \label{thm: minor-closed}
If $\Gamma/\Gamma_1$ is trivial or isomorphic to a subgroup in $\cA$, then $\cM(\Gamma, \Gamma_1, \cA)$ is minor-closed.
In particular, $\cM(\Gamma, \Gamma_1, \cA)$ is minor-closed if $\Gamma$ is finite.
\end{theorem}

This was proved by Anderson, Su, and Zaslavsky \cite[Theorem 9.6]{Anderson-Su-Zaslavsky2024} in the special case that $\Gamma = \Gamma_1 \rtimes \{1, -1\}^{\times}$ for an abelian group $\Gamma$ with no elements of order $2$.
Recall from Problem \ref{prob: infinite group partition properties}$(iii)$ that $\Gamma/\Gamma_1$ may always be trivial or isomorphic to a subgroup in $\cA$; this is an open problem.

An \emph{excluded minor} for a minor-closed class $\cM$ of matroids is a matroid that is not in $\cM$ but has every proper minor in $\cM$.
There are several interesting results and conjectures concerning the excluded minors for the classes of frame matroids, lifted-graphic matroids, $\Gamma$-frame matroids, and $\Gamma$-lifted-graphic matroids.
We will briefly survey these results, and then make two conjectures about the excluded minors of $\cM(\Gamma, \Gamma_1, \cA)$.

Chen and Geelen \cite{Chen-Geelen-2018} proved that the classes of frame matroids and lifted-graphic matroids have infinitely many excluded minors.
This has the following consequence.

\begin{proposition}
The class of elementary lifts of frame matroids has infinitely many excluded minors.
\end{proposition}
\begin{proof}
Let $N$ be an excluded minor for the class of frame matroids.
We will show that $N \oplus N$ is an excluded minor for the class of elementary lifts of frame matroids, which will prove the statement.
First, for any $e \in E(N)$, the matroids $N\del e \oplus N$ and $N/e \oplus N$ are elementary lifts of frame matroids because $N/f$ is a frame matroid for any $f \in E(N)$ and the direct sum of two frame matroids is a frame matroid.
Suppose that there is a matroid $K$ with an element $e$ so that $K \del e = N \oplus N$ and $K/e$ is a frame matroid.
Since $N$ is not a frame matroid, both copies of $E(N)$ span $e$.
But then $e$ is a loop, which contradicts that $K/e$ is a frame matroid.
Therefore $N \oplus N$ is not an elementary lift of a frame matroid.
\end{proof}

DeVos, Funk, and Pivotto \cite[Theorem 1.6]{DeVos-Funk-Pivotto-2014} proved that if $\Gamma'$ is infinite, then the classes of $\Gamma'$-frame matroids and $\Gamma'$-lifted-graphic matroids have infinitely many excluded minors.
We expect that this is true more generally for $\cM(\Gamma, \Gamma_1, \cA)$ when $\Gamma$ is infinite, and that the following conjecture could be proved by modifying the techniques from \cite{DeVos-Funk-Pivotto-2014}.

\begin{conjecture} \label{conj: excluded minors infinite case}
If $\Gamma$ is infinite and $\Gamma/\Gamma_1$ is isomorphic to a group in $\cA$, then $\cM(\Gamma, \Gamma_1, \cA)$ has infinitely many excluded minors.
\end{conjecture}

What about finite groups?
It is expected that if $\Gamma'$ is finite and abelian, then the classes of $\Gamma'$-frame and $\Gamma'$-lifted-graphic matroids have finitely many excluded minors; this would follow from a conjecture of Geelen and Gerards \cite[Conjecture 1.1]{Geelen-Gerards-2009}.
We are not aware of any results or conjectures for the case in which $\Gamma'$ is finite and non-abelian.
However, we make the optimistic conjecture that if $\Gamma$ is finite, then $\cM(\Gamma, \Gamma_1, \cA)$ has finitely many excluded minors.

\begin{conjecture} \label{conj: excluded minors finite case}
If $\Gamma$ is finite, then $\cM(\Gamma, \Gamma_1, \cA)$ has finitely many excluded minors.
\end{conjecture}

If true, this is likely to be extremely difficult to prove. 
In the special case that $|\Gamma| = |\Gamma_1| = 2$, $\cM(\Gamma, \Gamma_1, \cA)$ is the class of \emph{even-cycle matroids}, and as described in \cite{Grace-vanZwam-2019}, Pivotto and Royle have found nearly $400$ excluded minors for this class.

\subsection{Duality and direct sums}

The class $\cM(\Gamma, \Gamma_1, \cA)$ is never closed under duality.
There are many graphic matroids with dual not in $\cM(\Gamma, \Gamma_1, \cA)$ for any $\Gamma$, $\Gamma_1$, and $\cA$.

\begin{proposition} \label{prop: duality}
The class $\cM(\Gamma, \Gamma_1, \cA)$ is not closed under duality.
In particular, if $G$ is a $2$-connected graph that is not projective-planar, then $M(G) \oplus M(G)$ is in  $\cM(\Gamma, \Gamma_1, \cA)$ and $(M(G) \oplus M(G))^*$ is not in $\cM(\Gamma, \Gamma_1, \cA)$.
\end{proposition}
\begin{proof}
Let $G'$ be the disjoint union of two copies of $G$, and note that $M(G') \cong M(G) \oplus M(G)$.
Let $\psi$ be the $\Gamma$-gain function for $G'$ that maps each oriented edge to the identity of $\Gamma$.
Then $M(\Gamma, \Gamma_1, \cA, G', \psi) = M(G')$, so $M(G) \oplus M(G)$ is in  $\cM(\Gamma, \Gamma_1, \cA)$.

Note that $(M(G) \oplus M(G))^*$ is equal to $M(G)^* \oplus M(G)^*$ \cite[Proposition 4.2.21]{Oxley2011}.
We will first show that $M(G)^*$ is not a frame matroid.
Suppose it is a frame matroid.
Since $M(G)^*$ is binary, a result of Zaslavsky \cite[Corollary 3.2]{Zaslavsky-1987} shows that $M(G)^*$ is signed-graphic.
Then since $G$ is $2$-connected and $M(G)^*$ is signed-graphic, a result of Slilaty \cite[Theorem 3]{Slilaty-2005} implies that $G$ is projective-planar, a contradiction.
So $M(G)^*$ is not a frame matroid.

Now suppose that $M(G)^* \oplus M(G)^*$ is in $\cM(\Gamma, \Gamma_1, \cA)$.
Then there is a matroid $K$ with an element $e$ so that $K \del e = M(G)^* \oplus M(G)^*$ and $K/e$ is a frame matroid.
Since $M(G)^*$ is not a frame matroid, both copies of $E(M(G)^*)$ in $K$ span $e$.
But then $e$ is a loop, a contradiction.
\end{proof}

The class $\cM(\Gamma, \Gamma_1, \cA)$ is never closed under direct sums, unless $\Gamma_1$ is trivial, in which case $\cM(\Gamma, \Gamma_1, \cA)$ is the class of $\Gamma$-frame matroids.
For an integer $r \ge 3$, an \emph{$r$-spike} is a simple rank-$r$ matroid on $2r+1$ elements with an element $t$ (a \emph{tip}) that is on $r$ lines, each with three elements \cite[Proposition 1.5.18]{Oxley2011}.

\begin{proposition} \label{prop: direct sums}
If $\Gamma_1$ is nontrivial, then $\cM(\Gamma, \Gamma_1, \cA)$ is not closed under direct sums.
In particular, $\cM(\Gamma, \Gamma_1, \cA)$ contains an $r$-spike for all $r \ge 3$, and if $M$ is an $r$-spike in $\cM(\Gamma, \Gamma_1, \cA)$ with $r \ge 5$, then $M \oplus M$ is not in $\cM(\Gamma, \Gamma_1, \cA)$.
\end{proposition}
\begin{proof}
We will first show that $\cM(\Gamma, \Gamma_1, \cA)$ contains an $r$-spike for all $r \ge 3$.
Let $G$ be the graph obtained from an $r$-edge cycle by making a parallel copy of each edge and then adding a loop to one vertex.
Let $\alpha$ be a non-identity element of $\Gamma_1$, and let $\psi$ be a $\Gamma$-gain function for $G$ that maps each oriented edge to $\ep$ or $\alpha$ so that each loop and $2$-cycle is unbalanced.
Let $M_1 = M(\Gamma, \Gamma_1, \cA, G, \psi)$.
Since the loop of $G$ together with each parallel pair of edges forms a circuit of $M_1$ we see that $M_1$ is an $r$-spike with the loop as the tip.

Let $M$ be any $r$-spike in $\cM(\Gamma, \Gamma_1, \cA)$ with $r \ge 5$.
It was proven in \cite[page 35]{Geelen-Nelson-Walsh-2024} that $r$-spikes with $r \ge 5$ are not frame matroids.
Suppose that $M  \oplus M \in \cM(\Gamma, \Gamma_1, \cA)$.
Then there is a matroid $K$ with an element $e$ so that $K \del e = M \oplus M$ and $K/e$ is a frame matroid.
Since $M$ is not a frame matroid, both copies of $E(M)$ in $K$ span $e$.
But then $e$ is a loop, a contradiction.
\end{proof}

\section{Examples} \label{sec: examples}

In this section we will consider three interesting classes of elementary lifts of frame matroids, two from existing literature and one new, and show that they arise as special cases of the matroid of Definition \ref{def: the matroid}.

\subsection{Symmetry-forced rigidity matroids} \label{sec: rigidity}

For our first example we will show that a recent construction of Bernstein \cite{Bernstein2022} arises from Definition \ref{def: the matroid} in the special case that $\Gamma$ is $SE(2)$, the special Euclidean group of direct isometries of $\bR^2$.
Let $(G, \psi)$ be an $SE(2)$-gain graph.
We will first describe the matroid of Theorem \ref{thm: main} applied to $(G, \psi)$, then describe Bernstein's matroid for $(G, \psi)$, and then show that they are equal.

Recall from Example \ref{ex: special Euclidean group as Frobenius group} that $SE(2)$ is isomorphic to the semidirect product $T(2) \rtimes SO(2)$, where $T(2)$ is the normal subgroup of translations and $SO(2)$ is the subgroup of rotations.
Also recall that if $\cA$ is the set of point-stabilizing subgroups of $SE(2)$, then $\{T(2)\} \cup \cA$ is a Frobenius partition of $SE(2)$.
By Lemma \ref{lem: equivalences for the linear class} with $(\Gamma, \Gamma_1, \cA) = (SE(2), T(2), \cA)$ we see that a circuit $C$ of the underlying $SO(2)$-frame matroid $F(G, \psi/T(2))$ is in the linear class $\cC(SE(2), T(2), \cA, G, \psi)$ of Definition \ref{def: the linear class} if $C$ is an $SE(2)$-balanced cycle or \begin{enumerate}

    \item[$(*)$] $C$ is a handcuff or a theta graph with a cyclic covering pair $(W_1, W_2)$ of walks such that $\psi(W_1)$ and $\psi(W_2)$ fix the same point of $\bR^2$.
\end{enumerate} 
We will write $\cC$ for $\cC(SE(2), T(2), \cA, G, \psi)$, and $M$ for $M(SE(2), T(2), \cA, G, \psi)$.

We next describe how Bernstein defines a linear class $\cC'$ of the underlying $SO(2)$-frame matroid $F(G, \psi/T(2))$.
Following \cite[Definition 5.8]{Bernstein2022}, a \emph{covering pair of walks} of a theta graph or handcuff $C$ is a pair $(W_1, W_2)$ of closed walks on $G[C]$, starting at the same vertex, so that every edge is traversed one or two times in the concatenation $W_1W_2$.
Clearly every cyclic covering pair of walks, as defined in Section \ref{sec: the construction}, is a covering pair of walks, but the converse is not always true.
Then following \cite[Definition 5.8, Lemma 5.9, Lemma 5.16]{Bernstein2022}, a circuit $C$ of the underlying $SO(2)$-frame matroid $F(G, \psi/T(2))$ is in the linear class $\cC'$ if $C$ is an $SE(2)$-balanced cycle or
\begin{enumerate}
    \item[$(*')$] $C$ is a handcuff or a theta graph with a covering pair $(W_1, W_2)$ of walks such that $\psi(W_1)\psi(W_2) = \psi(W_2)\psi(W_1)$ and neither $\psi(W_1)$ nor $\psi(W_2)$ is a translation.\footnote{In \cite[Lemma 5.16]{Bernstein2022} the condition that neither $\psi(W_1)$ nor $\psi(W_2)$ is a translation is omitted, but it should be present or else $C$ could have a $(\psi/T(2))$-balanced cycle as a proper subset and would therefore not be a circuit of $F(G, \psi/T(2))$.}
\end{enumerate} 
Bernstein writes $M^L(G, \psi)$ for the elementary lift of $F(G, \psi/T(2))$, and uses it to give a combinatorial characterization of symmetry-forced rigidity over subgroups of $SE(2)$ \cite[Theorem 5.10]{Bernstein2022}.

To show that $M$ and $M^L(G, \psi)$ are equal it suffices to show that $\cC$ and $\cC'$ are equal.
Since they both contain the same cycles of $G$, it suffices to prove the following.

\begin{lemma} \label{lem: equivalence with Bernstein}
    Let $C$ be a handcuff or theta graph circuit of $F(G, \psi/T(2))$.
    Then $C \in \cC$ if and only if $C \in \cC'$.
\end{lemma}
\begin{proof}
Let $(W_1, W_2)$ be a cyclic covering pair of walks for $C$.
Then $C \in \cC$ if and only if $\psi(W_1)$ and $\psi(W_2)$ fix the same point of $\bR^2$ (by Lemma \ref{lem: equivalences for the linear class}), and $C \in \cC'$ if and only if $\psi(W_1)\psi(W_2) = \psi(W_2)\psi(W_1)$ and neither $\psi(W_1)$ nor $\psi(W_2)$ is a translation (by \cite[Lemma 5.9]{Bernstein2022}). 
So it suffices to show that two non-translations $\psi(W_1)$ and $\psi(W_2)$ in $SE(2)$ fix the same point if and only if they commute.
First suppose that $\psi(W_1)$ and $\psi(W_2)$ fix the same point $z$ of $\bR^2$.
The function that maps each isometry $Ax + b$ to $(Ax + b) - z$ is clearly an isomorphism of $SE(2)$.
The images of $\psi(W_1)$ and $\psi(W_2)$ under this isomorphism are both rotations, and therefore they commute.
This implies that $\psi(W_1)$ and $\psi(W_2)$ commute.
Conversely, suppose that $\psi(W_1)$ and $\psi(W_2)$ commute.
For $i = 1,2$ let $x_i \in \bR^2$ be fixed by $\psi(W_i)$ and let $A_i$ be the subgroup of isometries that fix $x_i$.
Since $\psi(W_1)\psi(W_2) = \psi(W_2)\psi(W_1)$ we have $\psi(W_2)^{-1}\psi(W_1)\psi(W_2) = \psi(W_1)$.
Since $A_i$ is a malnormal subgroup of $SE(2)$ this implies that $\psi(W_2) \in A_1$, so $x_1 = x_2$ and $\psi(W_1)$ and $\psi(W_2)$ fix the same point.
\end{proof}

Since Theorem \ref{thm: main} generalizes Bernstein's construction, one might hope that Theorem \ref{thm: main} may have applications in the theory of symmetry-forced rigidity.
However, it seems that $SE(2)$ is the only relevant group that satisfies the hypotheses of Theorem \ref{thm: main}, so it is unclear how to generalize Bernstein's construction in a way that is meaningful in rigidity theory.
We leave this as an open direction for future research.

\subsection{Matroids of gain signed graphs} \label{sec: gain signed graphs}

For our next example we will show that a recent construction of Anderson, Su, and Zaslavsky \cite{Anderson-Su-Zaslavsky2024} arises from Definition \ref{def: the matroid} in the special case that $\Gamma = \Gamma_1 \rtimes \{1,-1\}^{\times}$ for an (additively-written) abelian group $\Gamma_1$ with no elements of order $2$, where $-1$ acts on $\Gamma_1$ by inversion.\footnote{Anderson, Su, and Zaslavsky state that their construction applies with $\Gamma_1$ as any abelian group, but in fact it only applies when $\Gamma_1$ has no elements of order $2$. This follows from Theorem \ref{thm: main converse} together with the discussion of Section \ref{sec: gain signed graphs}, and one can also construct counterexamples with four vertices and six edges when $\Gamma_1$ has an element of order $2$.}
Let $(G, \psi)$ be a $(\Gamma_1 \rtimes \{1,-1\}^{\times})$-gain graph.
We will first describe the matroid of Theorem \ref{thm: main} for $(G, \psi)$, then describe the matroid of Anderson, Su, and Zaslavsky for $(G, \psi)$, and then show that they are equal.

Since $\Gamma_1 \rtimes \{1,-1\}^{\times}$ is a semidirect product, it has $\Gamma_1$ as a normal subgroup.
As described in Example \ref{ex: inverted infinite abelian group}, the set 
$$\cA = \big\{\{(0,1), (a,-1)\} \colon a \in \Gamma_1\big\}$$
is a collection of order-$2$ malnormal subgroups of $\Gamma_1 \rtimes \{1,-1\}^{\times}$ so that $\{\Gamma_1\} \cup \cA$ is a nontrivial Frobenius partition of $\Gamma_1 \rtimes \{1,-1\}^{\times}$.
We see from Definition \ref{def: the linear class} that a circuit $C$ of the underlying $\{1,-1\}^{\times}$-frame matroid $F(G, \psi/\Gamma_1)$ is in the linear class $\cC(\Gamma, \Gamma_1, \cA, G, \psi)$ if $C$ is a $\psi$-balanced cycle or $C$ is a handcuff or a theta graph that, up to switching equivalence, has a $\psi$-normalized spanning tree $T$ so that the two edges $e$ and $f$ in $C - T$ have $\psi(e) = \psi(f)$.

We next describe how Anderson, Su, and Zaslavsky define a matroid from $(G, \psi)$.
First, since Anderson et al. do not work explicitly with $(\Gamma_1 \rtimes \{1,-1\}^{\times})$-gain functions for $G$, we must show that their setting is equivalent to ours.\footnote{Up to half edges and loose edges, which they allow and we do not.}
They have a $\Gamma_1$-gain function $\phi$ and a function $\sigma\colon E(G) \to \{1,-1\}^{\times}$.
Given the pair $(\phi, \sigma)$ we can define a $(\Gamma_1 \rtimes \{1,-1\}^{\times})$-gain function $\psi$ by setting $\psi(e, u, v) = (\phi(e, u, v), \sigma(e))$ for each edge $e$ with ends $u$ and $v$.
Conversely, let $\psi$ be a $(\Gamma_1 \rtimes \{1,-1\}^{\times})$-gain function.
Let $\phi$ be the $\Gamma_1$-gain function with $\phi(e, u, v) = a$ if and only if $\psi(e, u, v) = (a, b)$, and let $\sigma$ be the function from $E(G)$ to $\{1,-1\}^{\times}$ so that each edge $e$ with ends $u$ and $v$ satisfies $\sigma(e) = b$ if $\psi(e, u, v) = (a, b)$ (this is well-defined because each element in $\{1,-1\}^{\times}$ is its own inverse).
Then $\phi$ is a $\Gamma_1$-gain function and $\sigma$ is a function from $E(G)$ to $\{1,-1\}^{\times}$, and it is straightforward to see that each edge $e$ with ends $u$ and $v$ satisfies $(\phi(e, u, v), \sigma(e)) = \psi(e, u, v)$.
So pairs $(\phi, \sigma)$ are in bijection with $(\Gamma_1 \rtimes \{1,-1\}^{\times})$-gain functions.

Given a pair $(\phi, \sigma)$ where $\phi$ is a $\Gamma_1$-gain function for $G$ and $\sigma \colon E(G) \to \{1, -1\}^{\times}$, we will next describe how Anderson et al. assign to each walk $W$ in $G$ a value in $\Gamma_1$, denoted $\phi(W)$.
Let $W = u_0,e_1,u_1,e_2,\dots,e_l,u_l$ be a walk in $G$ with $l \ge 1$.
Anderson et al. \cite[pg. 514, (2.1)]{Anderson-Su-Zaslavsky2024} define
\begin{align}
\setcounter{equation}{0}
\phi(W) = \sum_{i = 1}^{l}\phi(e_i, u_{i-1}, u_i)\sigma(W_{0,i-1}),
\end{align}
where $\sigma(W_{0,j}) = \prod_{i = 1}^{j}\sigma(e_i)$.\footnote{Anderson et al. also have an orientation $\tau$ of $(G, \sigma)$, but since $\phi(W)$ is invariant under reorientation by \cite[Lemma 2.5]{Anderson-Su-Zaslavsky2024}, we are setting $\tau(u_{i-1}, e_i) = -1$ for all $i \in [l]$.}
We will proceed by induction on $l$ to show that if $\psi$ is the $(\Gamma_1 \rtimes \{1,-1\})$-gain function for $G$ with $\psi(e, u, v) = (\phi(e, u, v), \sigma(e))$ for each oriented edge $(e, u, v)$, then $\psi(W) = (\phi(W), \sigma(W))$.
This is clearly true if $l = 0$, so we may assume that $l \ge 1$.
Let $W' = u_0,e_1,u_1,e_2,\dots,e_{l-1},u_{l-1}$.
Then
\begin{align}
\psi(W) &= \psi(W') \psi(e_l, u_{l-1}, u_l)\\
&= \big(\phi(W'), \sigma(W')\big)\big(\phi(e_l, u_{l-1}, u_l), \sigma(e_l)\big)\\
&= \big(\phi(W') + \sigma(W')\phi(e_l, u_{l-1}, u_l), \sigma(W')\sigma(e_l)\big) \\
&= \big(\phi(W') + \sigma(W')\phi(e_l, u_{l-1}, u_l), \sigma(W)\big) \\
&= \big(\phi(W), \sigma(W)\big).
\end{align}
Line (3) holds by the bijection between $\psi$ and $(\phi, \sigma)$ described in the previous paragraph, line (4) holds by the group operation of $\Gamma_1 \rtimes \{1, -1\}^{\times}$, and line (6) follows from (1).

We next describe how Anderson et al. use walks to define a special class of circuits of the $\{1,-1\}^{\times}$-frame matroid $F(G, \sigma)$.
Let $C$ be a circuit of $F(G, \sigma)$.
Following \cite[page 511]{Anderson-Su-Zaslavsky2024}, a \emph{circuit walk} for $C$ is a closed walk on $C$ that traverses every edge and traverses each edge of $C$ in a cycle exactly once.
Then following \cite[page 516]{Anderson-Su-Zaslavsky2024}, the \emph{gain} of $C$ is $\phi(W)$ for a circuit walk $W$ of $C$, and $C$ is \emph{neutral} if it has gain $0$ (this is well-defined by \cite[Proposition 2.8]{Anderson-Su-Zaslavsky2024}).
It is straightforward to check that every circuit walk $W$ satisfies $\sigma(W) = 1$.
Therefore, since $\psi(W) = (\phi(W), \sigma(W))$, a circuit of $(G, \psi)$ is neutral if and only if it has a circuit walk $W$ with $\psi(W) = (0,1)$.
We will show that the set of neutral circuits is precisely the linear class $\cC(\Gamma, \Gamma_1, \cA, G, \psi)$ from Definition \ref{def: the linear class}.

\begin{lemma}
Let $\Gamma_1$ be an abelian group with no elements of order $2$ and let $(G, \psi)$ be a $(\Gamma_1 \rtimes \{1,-1\}^{\times})$-gain graph. 
A circuit $C$ of the underlying frame matroid $F(G, \psi/\Gamma_1)$ is neutral if and only if it is in the linear class $\cC(\Gamma, \Gamma_1, \cA, G, \psi)$
\end{lemma}
\begin{proof}
Let $\cC'$ be the set of neutral circuits of $F(G, \psi/\Gamma_1)$ and let $\cC = \cC(\Gamma, \Gamma_1, \cA, G, \psi)$.
Let $C$ be a circuit of $F(G, \psi/\Gamma_1)$.
By previous discussion, $C$ is neutral if and only if it has a circuit walk $W$ with $\psi(W) = (0,1)$.
Since a cycle $C$ is in $\cC$ if and only if it has a simple closed walk $W$ with $\psi(W) = (0,1)$, this implies that $\cC'$ and $\cC$ contain the same cycles.
So we may assume that $C$ is a handcuff (as $F(G, \psi/\Gamma_1)$ has no theta graph circuits).
Let $v$ be a vertex of $G[C]$, and let $W$ be a circuit walk for $C$ starting at $v$.
By Lemma \ref{lem: normalize a forest} there is a switching function $\eta$ so that $\eta(v) = (0,1)$ and $C$ has a $\psi^{\eta}$-normalized spanning tree $T$.
Let $e_1,e_2 \in C - T$.
Recall that $C$ is neutral if and only if $\psi(W) = (0,1)$.
Since $\eta(v) = (0,1)$, this is true if and only if $\psi^{\eta}(W) = (0,1)$.
Since $T$ is $\psi^{\eta}$-normalized we see that $\psi^{\eta}(W) = (0,1)$ if and only if $\psi^{\eta}(e_1) = \psi^{\eta}(e_2)$.
By Lemma \ref{lem: equivalences for the linear class} we see that $\psi^{\eta}(e_1) = \psi^{\eta}(e_2)$ if and only if $C \in \cC$.
\end{proof}

Finally, we can describe how Anderson et al. use neutral circuits to define a matroid.
They define a function $r \colon E(G) \to \bZ^{\ge 0}$ by $r(X) = |V(G[X])| - b(X) + \delta(X)$, where $\delta(X) = 0$ if every circuit of $F(G, \psi/\Gamma_1)|X$ is neutral and $\delta(X) = 1$ otherwise \cite[Definition 4.1]{Anderson-Su-Zaslavsky2024}.
They then prove in \cite[Theorem 11.1]{Anderson-Su-Zaslavsky2024} that $r$ is the rank function of a matroid.
By comparing $r$ with the rank function from Proposition \ref{prop: rank function}, we see that since $\cC(\Gamma, \Gamma_1, \cA, G, \psi)$ is the class of neutral circuits, their matroid is equal to $M(\Gamma, \Gamma_1, \cA, G, \psi)$.

Since Theorem \ref{thm: main} generalizes the construction of Anderson et al., one might hope that Theorem \ref{thm: main} has applications in the study of hyperplane arrangements, as this study was part of the motivation for the construction of Anderson et al.
Based on the discussion in Sections 1 and 9 of \cite{Anderson-Su-Zaslavsky2024} this seems very promising. 
We leave this as an open direction for future research.

\subsection{Representable matroids} \label{sec: the representable case}

We will now show that important classes of representable matroids arise from Definition \ref{def: the matroid}.
Let $\bF$ be a field, let $\Gamma_1$ be a subgroup of $\bF^{+}$, and let $\Gamma_2$ be a subgroup of $\bF^{\times}$ so that $\Gamma_1$ is closed under scaling by elements in $\Gamma_2$.
A \emph{$\Gamma_2$-frame matrix} is a matrix so that each column has at most two nonzero entries, each nonzero column contains a $1$, and each column with two nonzero entries contains a $1$ and a $-\gamma$ for some $\gamma \in \Gamma_2$.
Consider the class $\cM$ of matroids with a representation over $\bF$ of the form shown in Figure \ref{fig: special representable matroids}.
This class is very important in matroid structure theory.
When $\bF$ is finite, matroids in $\cM$ are the building blocks for objects called \emph{frame templates} (see \cite[page 6]{Structure_Theory}), and Geelen, Gerards, and Whittle are in the process of proving that (roughly speaking) every proper minor-closed class of $\bF$-representable matroids can be described by a finite set of frame templates \cite[Theorem 4.2]{Structure_Theory}.
Also, work of Nelson and Walsh \cite{Nelson-Walsh2022} indicates that if $\bF$ is prime, then the matroids in $\cM$ will play a crucial role in determining the maximum number of elements of a simple rank-$r$ $\bF$-representable matroid with no $\PG(3, \bF)$-minor, a significant open problem in extremal matroid theory.

\begin{figure}
\begin{center}
\begin{tabular}{|ccc|}
\hline 
& a row with entries from $\Gamma_1$ & \\
\hline 
& & \\
& $\Gamma_2$-frame matrix & \\ [0.5cm]
\hline
\end{tabular}
\end{center}
\caption{If $\bF$ is a field and $\Gamma_1$ and $\Gamma_2$ are subgroups of $\bF^+$ and $\bF^{\times}$, respectively, so that that $\Gamma_1$ is closed under scaling by elements in $\Gamma_2$, then the class of matroids with an $\bF$-representation of the form shown above is minor-closed.}
\label{fig: special representable matroids}
\end{figure}

We will show that the matroids in $\cM$ arise from Definition \ref{def: the matroid}.
It suffices to consider the case in which $\Gamma_1 = \bF^+$ and $\Gamma_2 = \bF^{\times}$.
Recall from Examples \ref{ex: finite field semidirect product} and \ref{ex: infinite field affine transformations} that $\bF^+ \rtimes \bF^{\times}$ has group operation $\circ$ with $(a,b) \circ (c,d) = (a + bc, bd)$ and has a Frobenius partition $\{\bF^+\} \cup \cA$ where elements $(a,b)$ and $(c,d)$ are in the same set in $\cA$ if and only if $a/(1 - b) = c/(1 - d)$.
We will first define a natural incidence matrix over $\bF$ for any $(\bF^+ \rtimes \bF^{\times})$-gain graph $(G, \psi)$; see Figure \ref{fig: incidence matrix example} for an example.

\begin{definition} \label{def: incidence matrix}
Let $\bF$ be a field, let $(G, \psi)$ be an $(\bF^{+} \rtimes \bF^{\times})$-gain graph, and let $D$ be an orientation of $G$.
Define a matrix $A(D, \psi)$ over $\bF$ with rows indexed by $V(G) \cup \{v_0\}$ with $v_0 \notin V(G)$ and columns indexed by $E(G)$ as follows: if $u \in V(G) \cup \{v_0\}$ and $e$ is an edge of $G$ oriented in $D$ from $v$ to $w$ with $\psi(e, v, w) = (a, b)$ where $a \in \bF^+$ and $b \in \bF^{\times}$, then the entry $A(D,\psi)_{u,e}$ in row $u$ and column $e$ satisfies
$$A(D, \psi)_{u,e} = 
\begin{cases}
    a, \, \textrm{ if $u = v_0$} \\
    1, \, \textrm{ if $u = v$ and $u \ne w$} \\
    -b, \, \textrm{ if $u = w$ and $u \ne v$} \\
    1 - b, \, \textrm{ if $u = v = w$} \\
    0, \textrm{ otherwise.}
\end{cases}
$$
\end{definition}

\begin{figure}
    \centering
\begin{tabular}{c c ccc}
\setlength\tabcolsep{3cm}
\begin{tikzpicture}[scale=0.75,line cap=round,line join=round,>=triangle 45,x=1cm,y=1cm, decoration={markings, 
	mark= at position 0.55 with {\arrow[scale = 1.5]{stealth}}}]

\draw[postaction={decorate}] (0,1)-- (-2,-2);
\draw[postaction={decorate}]  (0,1)-- (2,-2);
\draw[postaction={decorate}]  (-2,-2)-- (2,-2);
\draw[postaction={decorate}]  (-2,-2) to[bend left=30] (0,1);

\begin{scriptsize}
\draw [fill=black] (0,1) circle (3.5pt);
\draw [fill=black] (-2,-2) circle (3.5pt);
\draw [fill=black] (2,-2) circle (3.5pt);
\draw[color=black] (-2.2, 0) node {\normalsize $(0,1)$};
\draw[color=black] (-0.3, -0.7) node {\normalsize $(1,3)$};
\draw[color=black] (0.1,-1.6) node {\normalsize $(0,2)$};
\draw[color=black] (1.6,0) node {\normalsize $(2,4)$};
\draw[color=black] (2.9, -0.9) node {\normalsize $(4,1)$};
\draw[color=black] (1.4, 1.5) node {\normalsize $(3,2)$};
\draw[color=black] (-0.4, 1.5) node {\large $v_1$};
\draw[color=black] (-2.5,-2.5) node {\large $v_2$};
\draw[color=black] (2.5,-2.5) node {\large $v_3$};
\end{scriptsize}


\path[min distance=1.5cm] (2,-2)edge[in=85,out=5] (2,-2);
\path[min distance=1.5cm] (0,1)edge[in=85,out=5] (0,1);
\end{tikzpicture}

& &&&

\raisebox{1.45\height}{$\begin{matrix}
    v_0 \\
    v_1 \\
    v_2\\
    v_3
    \end{matrix}
    \left[ \begin{array}{cccccc}
   3 & 1 & 0 & 2 & 0 & 4\\
   \hline
   \mm 1 & 1     & \mm 1 & 1     & 0 & 0\\
   0     & \mm 3 & 1     & 0     & 1 & 0\\
   0     & 0     & 0     & \mm 4 & \mm 2 & 0
   \end{array} \right]$}
\\

$(G, \psi)$ &&&& $A(D, \psi)$
\end{tabular}
    \caption{A $(\GF(5)^{+} \rtimes \GF(5)^{\times})$-gain graph $(G, \psi)$ with an orientation $D$ of $G$, and the incidence matrix $A(D, \psi)$ over $\GF(5)$.}
    \label{fig: incidence matrix example}
\end{figure}
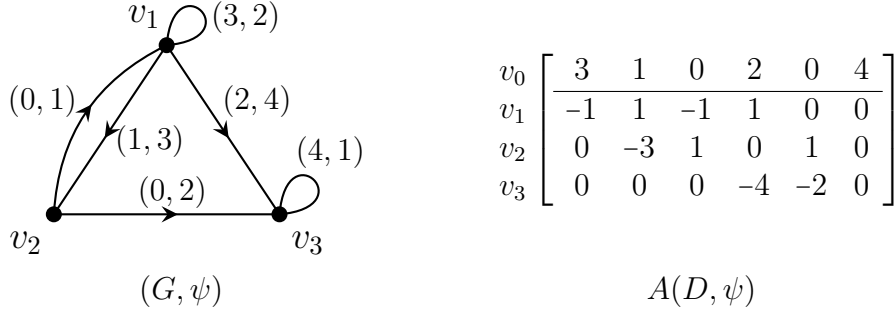

We will show that $A(D, \psi)$ represents the matroid $M = M(\bF^{+} \rtimes \bF^{\times}, \bF^+, \cA, G, \psi)$ over $\bF$.
We need two lemmas about operations that preserve the vector matroid of $A(D, \psi)$.
We will first show that if $D$ and $D'$ are two different orientations of $G$, then $A(D, \psi)$ and $A(D', \psi)$ are \emph{projectively equivalent}, which means that one can be obtained from the other by row operations and column scaling.
This implies that they have isomorphic vector matroids.

\begin{lemma} \label{lem: reorientation and projective equivalence}
Let $\bF$ be a field, let $(G, \psi)$ be an $(\bF^{+} \rtimes \bF^{\times})$-gain graph, and let $D$ and $D'$ be orientations of $G$.
Then $A(D', \psi)$ can be obtained from $A(D, \psi)$ by column scaling.
\end{lemma}
\begin{proof}
It suffices to consider the case in which there is only one edge $e$ that is oriented differently in $D$ and $D'$, so $e$ is a non-loop.
Let $u$ and $v$ be the ends of $e$, and assume that $e$ is oriented from $u$ to $v$ in $D$ and from $v$ to $u$ in $D'$.
Let $\psi(e, u, v) = (a, b)$.
Then column $e$ of $A(D, \psi)$ has entry $a$ in row $v_0$, entry $1$ in row $u$, entry $-b$ in row $v$, and all other entries are $0$.
Note that $(a, b)^{-1} = (-b^{-1}a, b^{-1})$ in $\bF^+ \rtimes \bF^{\times}$.
Let $A'$ be obtained from $A(D, \psi)$ by multiplying column $e$ by $-b^{-1}$.
In $A'$, column $e$ has entry $-b^{-1}a$ in row $v_0$, entry $-b^{-1}$ in row $u$, entry $1$ in row $v$, and all other entries are $0$.
Since $(a, b)^{-1} = (-b^{-1}a, b^{-1})$, we see that $A' = A(D', \psi)$.
\end{proof}

We will next show that applying a switching function to $\psi$ does not change the vector matroid of $A(D, \psi)$.

\begin{lemma} \label{lem: switching and projective equivalence}
Let $\bF$ be a field, let $G$ be a graph, and let $D$ be an orientation of $G$.
If $\psi$ and $\psi'$ are switching-equivalent $(\bF^{+} \rtimes \bF^{\times})$-gain functions for $G$, then $A(D, \psi)$ and $A(D, \psi')$ are projectively equivalent.
\end{lemma}
\begin{proof}
Let $A = A(D, \psi)$ and let $A' = A(D, \psi')$.
Let $v \in V(G)$ and let $(c, d) \in \bF^+ \rtimes \bF^{\times}$.
It suffices to show that if $\psi'$ is obtained from $\psi$ by switching at $v$ with value $(c,d)$, then $A$ and $A'$ are projectively equivalent.
By Lemma \ref{lem: reorientation and projective equivalence} we may assume that every non-loop of $G$ incident with $v$ is oriented towards $v$ in $D$.
Let $A_1$ be obtained from $A$ by the following sequence of operations:
\begin{itemize}
    \item Add $-c$ times row $v$ to row $v_0$.
    \item Multiply row $v$ by $d$.
    \item If $e$ is a loop at $v$, multiply column $e$ by $d^{-1}$.
\end{itemize}
We will show that $A_1 = A'$ by showing that for each edge $e$ of $G$, the columns of $A_1$ and $A'$ indexed by $e$ are equal.
Let $A_1[e]$ and $A'[e]$ be the columns of $A_1$ and $A'$, respectively, indexed by $e$.
If $e$ is not incident with $v$ then $A_1[e] = A[e] = A'[e]$, so we may assume that $e$ is incident with $v$.
First assume that $e$ is not a loop.
Let $u$ be the end of $e$ other than $v$, and let $\psi(e, u, v) = (a, b)$.
Then $A_1[e]$ has entry $a + bc$ in row $v_0$, entry $1$ in row $u$, entry $-bd$ in row $v$, and zeroes in all other rows.
Since $\psi'(e, u, v) = (a, b) \circ (c, d) = (a + bc, bd)$, we see that $A_1[e] = A'[e]$.
Next assume that $e$ is a loop at $v$.
Then $A_1[e]$ has entry $d^{-1}(a - c(1 - b))$ in row $v_0$, entry $1 - b$ in row $v$, and zeroes in all other rows.
Since 
\begin{align*}
\psi'(e, u, v) &= (c, d)^{-1} \circ (a, b) \circ (c, d) \\
&= (-d^{-1}c, d^{-1}) \circ (a + bc, bd) \\
& = (-d^{-1}c + d^{-1}(a + bc), d^{-1}bd) \\
&= (d^{-1}(a - c(1 - b)), b),
\end{align*}
we see that $A_1[e] = A'[e]$.
\end{proof}

We can now show that $A(D, \psi)$ represents the matroid $M(\bF^{+} \rtimes \bF^{\times}, \bF^+, \cA, G, \psi)$ over $\bF$.

\begin{theorem} \label{thm: matrix representation}
Let $\bF$ be a field, let $(G, \psi)$ be an $(\bF^{+} \rtimes \bF^{\times})$-gain graph, and let $D$ be an orientation of $G$.
Then $A(D, \psi)$ represents $M(\bF^{+} \rtimes \bF^{\times}, \bF^+, \cA, G, \psi)$ over $\bF$.
\end{theorem}
\begin{proof}
Let $A = A(D, \psi)$, and let $A'$ be obtained from $A$ by deleting the row indexed by $v_0$.
Then $M[A]$ is an elementary lift of $M[A']$.
Zaslavsky \cite[Theorem 2.1]{Zaslavsky2003} proved that $A'$ represents the underlying frame matroid $N = F(G, \psi/\bF^+)$ over $\bF$.
Therefore $M[A]$ is an elementary lift of $N$.
Let $\cC'$ be the linear class of circuits of $N$ that are also circuits of $M[A]$, and let $\cC = \cC(\bF^{+} \rtimes \bF^{\times}, \bF^+, \cA, G, \psi)$.
We will show that $\cC' = \cC$.
Let $C$ be a circuit of $N$, and let $T$ be a spanning tree of $G[C]$.
Since switching does not change $M$ (since $\cC$ is invariant under switching by definition) or $M[A]$ (by Lemma \ref{lem: switching and projective equivalence}) we may assume that $T$ is $\psi$-normalized.

First suppose that $C$ is a cycle of $G$.
Let $f$ be the unique edge in $C - T$.
If $C \in \cC$, then $C$ is $\psi$-normalized, and so for each $e \in C$ the entry of $A$ in row $v_0$ and column $e$ is $0$.
Then clearly $C$ has the same rank in $N$ and $M[A]$, so $C \in \cC'$.
Conversely, if $C \in \cC'$, then $C$ has the same rank in $N$ and $M[A]$ if and only the entry of $A$ in row $v_0$ and column $f$ is $0$.
Therefore $C$ is $\psi$-normalized, so $C \in \cC$.

Next suppose that $C$ is a handcuff or a theta graph.
Let $f$ and $g$ be the two edges in $C - T$.
Let $(a,b) \in \psi(f)$ and let $(c,d) \in \psi(g)$.
Recall from Example \ref{ex: infinite field affine transformations} that  $C \in \cC$ if and only if $a/(1 - b) = c/(1 - d)$.
Note that $C \in \cC'$ if and only if the columns of $A[C]$ are linearly dependent.
Therefore the following claim shows that $C \in \cC'$ if and only if $C \in \cC$.

\begin{claim}
The columns of $A[C]$ are linearly dependent if and only if $a/(1-b) = c/(1-d)$.
\end{claim}
\begin{proof}
Note that if $C'$ is a proper subset of $C$, then $C'$ is an independent set in $N$ and therefore the columns of $A'[C']$ are linearly independent, which implies that the columns of $A[C']$ are linearly independent.
So the columns of $A[C]$ are linearly independent if and only if there is a linear combination $\sum_{e \in C} x_e  \cdot A[e]$ which is equal to zero with each $x_e$ nonzero.
We will consider three cases depending on whether $C$ is a tight handcuff, a loose handcuff, or a theta graph.

First suppose that $C$ is a tight handcuff.
Let $u$ be the vertex of $G[C]$ with degree greater than $2$.
We may assume that $f$ and $g$ are incident with $u$.
Let $P_1$ and $P_2$ be the two cycles in $G[C]$ so that $f \in P_1$ and $g \in P_2$.
By reorienting $D$ we may assume that $P_1$ and $P_2$ form directed cycles in $D$ so that $f$ and $g$ are oriented away from $u$ if they are non-loops.
Let $(a,b)$ and $(c,d)$ be the $\psi$-values of the orientations of $f$ and $g$, respectively, away from $u$.
The columns of $A[C]$ are linearly independent if and only if there is a linear combination $\sum_{e \in C} x_e  \cdot A[e]$ which is equal to zero with each $x_e$ nonzero.
This is true if and only if each of the following holds, where $(iv)$ uses $(ii)$ and $(iii)$:
\begin{enumerate}[label=$(\roman*)$]
    \item $x_fa + x_gc = 0$. (Dependence in row $v_0$.)

    \item $x_e = bx_f$ for each $e \in P_1 - f$. (Dependence in the rows in $V(P_1) - \{u\}$.)

    \item $x_e = dx_g$ for each $e \in P_2 - g$. (Dependence in the rows in $V(P_2) - \{u\}$.)

     \item $x_f(1 - b) + x_g(1 - d) = 0$. (Dependence in row $u$, even if $f$ or $g$ is a loop of $G$.)
\end{enumerate}
If $a/(1-b) = c/(1-d)$, then we can set $(x_f, x_g) = (c, -a)$, set $x_e = bc$ if $e \in P_1 - f$, and set $x_e = -ad$ if $e \in P_2 - g$, and we see that equations $(i)$--$(iv)$ hold.
Conversely, suppose that equations $(i)$--$(iv)$ hold.
By $(i)$ we see that $-x_g/x_f = a/c$ and by $(iv)$ we see that $-x_g/x_f = (1 - b)/(1 - d)$, so $a/(1-b) = c/(1 - d)$, as desired.

In the second case suppose that $C$ is a loose handcuff.
Let $u$ and $v$ be the vertices of $G[C]$ with degree greater than $2$.
We may assume that $f$ is incident with $u$ and $g$ is incident with $v$.
Let $P_1$ and $P_2$ be the two cycles in $G[C]$ so that $f \in P_1$ and $g \in P_2$, and let $P_3$ be the path in $G[C]$ between $u$ and $v$.
Let $h$ be the edge of $P_3$ incident with $u$.
By reorienting $D$ we may assume that $P_1$ and $P_2$ form directed cycles in $D$ so that $f$ and $g$ are oriented away from $u$ and $v$, respectively (if they are non-loops), and $P_3$ is a directed path from $u$ to $v$.
Let $(a,b)$ and $(c,d)$ be the $\psi$-values of the orientations of $f$ and $g$ away from $u$ and $v$, respectively.
The columns of $A[C]$ are linearly independent if and only if there is a linear combination $\sum_{e \in C} x_e  \cdot A[e]$ which is equal to zero with each $x_e$ nonzero.
This is true if and only if each of the following holds, where $(v)'$ uses $(ii)'$, and $(vi)'$ uses $(iii)'$ and $(iv)'$:
\begin{enumerate}[label=$(\roman*)'$]
    \item $x_fa + x_gc = 0$. (Dependence in row $v_0$.)

    \item $x_e = bx_f$ for each $e \in P_1 - f$. (Dependence in the rows in $V(P_1) - \{u\}$.)

    \item $x_e = dx_g$ for each $e \in P_2 - g$. (Dependence in the rows in $V(P_2) - \{u\}$.)

    \item $x_e = x_h$ for each $e \in P_3 - h$. (Dependence in the rows in $V(P_3) - \{h\}$.)

     \item $x_f(1 - b) + x_h = 0$. (Dependence in row $u$, even if $f$ is a loop of $G$.)

     \item $x_g(1 - d) - x_h = 0$. (Dependence in row $v$, even if $g$ is a loop of $G$.)
\end{enumerate}
If $a/(1-b) = c/(1-d)$, then we can set $(x_f, x_g, x_h) = (c, -a, -a(1 - d))$, set $x_e = bc$ if $e \in P_1 - f$, set $x_e =-ad$ if $e \in P_2 - g$, and set $x_e = x_h$ if $e \in P_3 - h$, and we see that equations $(i)'$--$(vi)'$ hold.
Conversely, suppose that equations $(i)'$--$(vi)'$ hold.
By $(i)'$ we see that $-x_g/x_f = a/c$.
By summing $(v)'$ and $(vi)'$ we see that $x_f(1 - b) + x_g(1 - d) = 0$, so $-x_g/x_f = (1 - b)/(1 - d)$, and therefore $a/(1 - b) = c/(1 - d)$, as desired.

Finally, suppose that $C$ is a theta graph.
Let $u$ and $v$ be the vertices of $G[C]$ with degree $3$.
We may assume that $f$ and $g$ are incident with $u$ and that the $\psi$-values of the orientations of $f$ and $g$ away from $u$ are $(a,b)$ and $(c,d)$, respectively.
Let $P_1$, $P_2$, and $P_3$ be the three paths in $G[C]$ from $u$ to $v$ so that $f \in P_1$ and $g \in P_2$.
Let $h$ be the edge of $P_3$ that is incident with $u$.
We may assume that $D$ is oriented so that each $P_i$ is a directed path from $u$ to $v$ in $D$.
The columns of $A[C]$ are linearly independent if and only if there is a linear combination $\sum_{e \in C} x_e  \cdot A[e]$ which is equal to zero with each $x_e$ nonzero.
This is true if and only if each of the following holds:
\begin{enumerate}[label=$(\roman*)''$]
    \item $x_fa + x_gc = 0$. (Dependence in row $v_0$.)

    \item $x_f + x_g + x_h = 0$. (Dependence in row $u$.)

    \item $x_e = bx_f$ for each $e \in P_1 - f$. (Dependence in the rows in $V(P_1) - \{u,v\}$.)

    \item $x_e = dx_g$ for each $e \in P_2 - g$. (Dependence in the rows in $V(P_2) - \{u,v\}$.

    \item $x_e = x_h$ for each $e \in P_3 - h$. (Dependence in the rows in $V(P_3) - \{u,v\}$.)

    \item $-bx_f - dx_g - x_h = 0$. (Dependence in row $v$.)
\end{enumerate}
If $a/(1-b) = c/(1-d)$, then we can set $(x_f, x_g, x_h) = (c, -a, a - c)$, set $x_e = bc$ if $e \in P_1 - f$, set $x_e = -ad$ if $e \in P_2 - g$, and set $x_e = a - c$ if $e \in P_3 - h$, and we see that equations $(i)''$--$(vi)''$ all hold.
Conversely, suppose that equations $(i)''$--$(vi)''$ all hold.
By $(i)''$ we see that $-x_g/x_f = a/c$.
By summing $(ii)''$ and $(vi)''$ we see that $x_f(1 - b) + x_g(1 - d) = 0$, so $-x_g/x_f = (1 - b)/(1 - d)$.
Therefore $a/b = (1-b)/(1-d)$, so $a/(1 - b) = c/(1 - d)$, as desired.
\end{proof}

For all circuits $C$ of $N$ we have shown that $C \in \cC$ if and only if $C \in \cC'$.
Therefore $\cC = \cC'$ and it follows that $M(\bF^{+} \rtimes \bF^{\times}, \bF^+, \cA, G, \psi) = M[A]$.
\end{proof}

Theorem \ref{thm: matrix representation} recovers several known results in special cases.
Let $M = M(\bF^{+} \rtimes \bF^{\times}, \bF^+, \cA, G, \psi)$.
If the image of $\psi$ is the identity of $\bF^+\rtimes \bF^{\times}$, then $M$ is the graphic matroid of $G$ and $A(D, \psi)$ is simply the incidence matrix of $D$, which is known to represent $M$ (see \cite[Section 5.1]{Oxley2011}).
If $\psi$ is an $\bF^+$-gain function for $G$, then $M$ is the lift matroid of $(G, \psi)$, and Zaslavsky \cite[Theorem 4.1]{Zaslavsky2003} proved that $A(D, \psi)$ represents $M$.
Similarly, if $\psi$ is an $\bF^{\times}$-gain function for $G$, then $M$ is the frame matroid of $(G, \psi)$, and Zaslavsky  \cite[Theorem 2.1]{Zaslavsky2003} proved that $A(D, \psi)$ represents $M$.
Finally, if $\bF$ does not have characteristic $2$ and the image of $\psi$ lies in the subgroup $\bF^+ \rtimes \{1,-1\}^{\times}$, then Anderson, Su, and Zaslavsky \cite[Section 3]{Anderson-Su-Zaslavsky2024} proved that $A(D, \psi)$ represents $M$.
This special case also arises in work of Tanigawa \cite[Section 9]{Tanigawa-2015}, who constructed the matrix $A(D, \psi)$ from $(G, \psi)$ as a special case of a more general construction with applications in rigidity theory.
(Tanigawa's construction uses a subgroup $\Gamma = \Gamma_2 \rtimes \Gamma_1$ of $\bF^d \rtimes GL(\bF^d)$ and a bilinear form $b \colon \bF^d \rtimes \bF^d \to \bF^k$ with $k \ge 1$ such that $b(gx, y) = b(x, g^{-1}y)$ for all $x,y \in \bF^d$ and all $g \in \Gamma_1$.
When $d = 1$, $GL(\bF) \cong \bF^{\times}$, and the condition that $b(gx, y) = b(x, g^{-1}y)$ implies that $g^{-1} = g$ for all $g \in \Gamma_1$.
Therefore $\Gamma_1$ is either trivial, or $\Gamma_1 = \{1,-1\}^{\times}$.)

There is a beautiful theory, conjectured by Zaslavsky \cite{Zaslavsky2003} and proved by Funk, Pivotto, and Slilaty \cite{Funk-Pivotto-Slilaty-2022}, about the relationship between representations of a frame matroid $F(G, \cB)$ or lifted-graphic matroid $L(G, \cB)$ over a field and switching classes of gain functions over that field that give rise to $(G, \cB)$.
We conjecture that much of this theory extends to representations of elementary lifts of frame matroids.
We will first need to extend the concept of scaling from $\bF^+$-gain functions to $(\bF^+ \rtimes \bF^{\times})$-gain functions.

\begin{definition} \label{def: scaling}
Let $G$ be a graph, let $\bF$ be a field, and let $\psi$ and $\psi'$ be two $(\bF^+ \rtimes \bF^{\times})$-gain functions for $G$.
We say that $\psi$ and $\psi'$ are \emph{scaling equivalent} if there is some $c \in \bF^{\times}$ so that for every oriented edge $(e, u, v)$, if $\psi(e, u, v) = (a, b)$ then $\psi'(e, u, v) = (ac, b)$.
If this is the case then we write $\psi' = c\cdot \psi$.
If $\psi''$ can be obtained from $\psi$ by a sequence of scalings and vertex switchings, then $\psi$ and $\psi''$ are \emph{switching-and-scaling equivalent}.
\end{definition}

Note that if $c \in \bF^{\times}$ and $D$ is any orientation of $G$, then $A(D, \psi)$ and $A(D, c \cdot \psi)$ are projectively equivalent, because $A(D, c \cdot \psi)$ is obtained from $A(D, \psi)$ by multiplying row $v_0$ by $c$.
It then follows from Lemma \ref{lem: switching and projective equivalence} that if $\psi$ and $\psi'$ are switching-and-scaling equivalent, then $A(D, \psi)$ and $A(D, \psi')$ are projectively equivalent, so $M(\bF^{+} \rtimes \bF^{\times}, \bF^+, \cA, G, \psi) = M(\bF^{+} \rtimes \bF^{\times}, \bF^+, \cA, G, \psi')$ by Theorem \ref{thm: matrix representation}.
We conjecture that the converse is true as well.

\begin{conjecture} \label{conj: equivalence for canonical matrices}
Let $\bF$ be a field, let $G$ be a loopless graph with orientation $D$, and let $\psi$ and $\psi'$ be $(\bF^+\rtimes \bF^{\times})$-gain functions for $G$ so that $M(\bF^+ \rtimes \bF^{\times}, \bF^+, \cA, G, \psi) = M(\bF^+ \rtimes \bF^{\times}, \bF^+, \cA, G, \psi')$, and this matroid is $3$-connected and has rank at least $3$.
Then $A(D, \psi)$ and $A(D, \psi')$ are projectively equivalent if and only if $\psi$ and $\psi'$ are switching-and-scaling equivalent. 
\end{conjecture}

If true, this would imply Theorem 1 from \cite{Funk-Pivotto-Slilaty-2022}, which proves the special cases that $\psi$ and $\psi'$ are both $\bF^{\times}$-gain functions or both $\bF^+$-gain functions.
The conditions that $G$ is loopless and $M$ is $3$-connected and has rank at least $3$ are necessary for the same reasons as described in \cite{Funk-Pivotto-Slilaty-2022}.

Let $M = M(\bF^+ \rtimes \bF^{\times}, \bF^+, \cA, G, \psi)$.
So far we have only considered representations of $M$ given by incidence matrices of $(\bF^+ \rtimes \bF^{\times})$-gain graphs.
Up to projective equivalence, are these the only $\bF$-representations of $M$?
We conjecture an affirmative answer.
In fact, we conjecture that this is true for any $3$-connected, $\bF$-representable elementary lift of a frame matroid.

\begin{conjecture} \label{conj: representable implies incidence matrix of gain function}
Let $(G, \cB)$ be a biased graph, let $M$ be an elementary lift of the frame matroid $F(G, \cB)$ so that $M$ is $3$-connected, and let $\bF$ be a field.
If $A$ is a matrix representing $M$ over $\bF$, then there is an $(\bF^+ \rtimes \bF^{\times})$-gain function $\psi$ for $G$ so that $F(G, \cB) = F(G, \psi/\bF^+)$, $M = M(\bF^+ \rtimes \bF^{\times}, \bF^+, \cA, G, \psi)$, and $A$ is projectively equivalent to $A(D, \psi)$ for any orientation $D$ of $G$.
\end{conjecture}

If true, this would be an analogue of Theorem 2 from \cite{Funk-Pivotto-Slilaty-2022}.
Together, Conjectures \ref{conj: equivalence for canonical matrices} and \ref{conj: representable implies incidence matrix of gain function} would imply that for every $3$-connected elementary lift $M$ of a frame matroid $F(G, \cB)$ with $G$ loopless and every field $\bF$, there is a one-to-one correspondence between projective equivalence classes of $\bF$-representations of $M$ and switching-and-scaling classes of $(\bF^{\times} \times \bF^+)$-gain functions for $G$ that give rise to $M$.
The conjectures will likely be approachable using the techniques from \cite{Funk-Pivotto-Slilaty-2022}.
We leave this as an open direction for future research.

\section{A partial converse} \label{sec: converse}

In this section we will prove Theorem \ref{thm: main converse}.
Recall that $(K_n^{\Gamma}, \psi_n^{\Gamma})$ is the complete $n$-vertex $\Gamma$-gain graph as defined in Section \ref{sec: gain graphs}, and $\alpha_{ij}$ denotes the edge oriented from vertex $i$ to vertex $j$ and labeled by $\alpha$.
The setting is that we have a finite group $\Gamma$ with a normal subgroup $\Gamma_1$, and an elementary lift $M$ of the  $(\Gamma/\Gamma_1)$-frame matroid $F(K_n^{\Gamma}, \psi_n^{\Gamma}/\Gamma_1)$ so that a cycle of $K_n^{\Gamma}$ is a circuit of $M$ if and only if it is $\psi_n^{\Gamma}$-balanced.
We will show that $M$ arises from Definition \ref{def: the matroid}.

We will begin by showing that the linear class for $M$ as an elementary lift of $F(K_n^{\Gamma}, \psi_n^{\Gamma}/\Gamma_1)$ is invariant under the natural action of switching on the edge set of $K_n^{\Gamma}$.
More precisely, if $\eta \colon [n] \to \Gamma$ is a switching function, then $\eta$ induces a permutation of the edge set of $K_n^{\Gamma}$ where the edge $\alpha_{ij}$ with $i < j$ maps to the edge $(\eta^{-1}(v_i) \circ \alpha \circ \eta(v_j))_{ij}$.
For $X \subseteq E(K_n^{\Gamma})$ we write $\bar{\eta}(X)$ for the image of $X$ under the permutation of $E(K_n^{\Gamma})$ induced by $\eta$.
Note that the hypotheses of the following lemma are weaker than those of Theorem \ref{thm: main converse}: we allow unbalanced cycles to be in the linear class.

\begin{lemma} \label{lem: switching action}
Let $\Gamma$ be a finite group with a normal subgroup $\Gamma_1$ and let $n \ge 3$ be an integer.
Let $M$ be an elementary lift of the frame matroid $F(K_n^{\Gamma}, \psi_n^{\Gamma}/\Gamma_1)$ corresponding to a linear class $\cC$ so that every $\psi_n^{\Gamma}$-balanced cycle of $K_n^{\Gamma}$ is in $\cC$.
Then $\cC$ is invariant under the action of switching on $E(K_n^{\Gamma})$.
\end{lemma}
\begin{proof}
We will write $\circ$ and $\ep$ for the operation and identity, respectively, of $\Gamma$.
We will write $G$ for $K_n^{\Gamma}$ and $\psi$ for $\psi_n^{\Gamma}$, for convenience.
Let $N$ denote the underlying frame matroid $F(G, \psi/\Gamma_1)$.
We first prove a key claim about replacing a subpath with a path of the same $\psi$-value.

\begin{claim} \label{claim: replace a path}
Let $C \in \cC$ so that $C$ is not a cycle, let $P$ be a subpath of $C$ between distinct vertices $v_i$ and $v_j$ so that each internal vertex has degree $2$ in $C$, and let $V$ be the set of vertices of $C$ that are not internal vertices of $P$.
Let $P'$ be a path in $K_n^{\Gamma}$ from $v_i$ to $v_j$ so that no internal vertex of $P'$ is in $V$, and the walks along $P$ and $P'$ from $v_i$ to $v_j$ have the same value in $\Gamma$. 
If $C'$ is obtained from $C$ by replacing $P$ with $P'$, then $C' \in \cC$.
\end{claim}
\begin{proof}
First note that since each internal vertex of $P$ has degree $2$, the circuits $C$ and $C'$ are of the same type: a tight handcuff, loose handcuff, or a theta graph.
Since $C'$ contains no balanced cycles because the walks in $P$ and $P'$ from $v_i$ to $v_j$ have the same value in $\Gamma$ it follows that $C'$ is a circuit of $N$.
We now consider two cases.
First suppose that either $P$ or $P'$ consists of a single edge.
Then $P \cup P'$ is a $\psi$-balanced cycle.
Without loss of generality, we may assume that $P'$ consists of a single edge.
Consider the set $C \cup P'$.
The graph $G[C \cup P']$ has cyclomatic number $3$ and contains an unbalanced cycle, so $|C \cup P'| - r_N(C \cup P') \le 2$.
Since $C \cup P'$ is the union of the circuit $C$ and the $\psi$-balanced cycle $P \cup P'$, and $\cC$ is a linear class, it follows that every circuit of $N$ contained in $C \cup P'$ is in $\cC$, so $C' \in \cC$.
In the second case, suppose that $P$ and $P'$ each have at least two edges.
Let $P''$ be the edge oriented from $v_i$ and $v_j$ with the same $\psi$-value as the $v_i$-$v_j$ walks in $P$ and $P'$.
Let $C_1$ be obtained from $C$ by replacing $P$ with $P''$.
By the previous case, $C_1 \in \cC$.
Then $C'$ is obtained from $C_1$ by replacing $P''$ with $P'$, and once again by the previous case we have $C' \in \cC$ because $C_1 \in \cC$.
\end{proof}

Now let $C \in \cC$, and let $\eta \colon [n] \to \Gamma$ be a switching function for $(G, \psi)$.
Clearly $\bar{\eta}(C)$ is a circuit of $N$ because switching preserves graph isomorphism (so $\bar{\eta}(C)$ has the same type as $C$: cycle, tight handcuff, loose handcuff, or theta graph) and the set of $\psi$-balanced cycles.
If $C$ is a $\psi$-balanced cycle then so is $\bar{\eta}(C)$ and therefore $\bar{\eta}(C) \in \cC$ by hypothesis, so we may assume that $C$ is not a cycle.
Since every switching function can be written as a composition of switching functions that switch with a non-identity value at only one vertex, we may assume without loss of generality that $\eta(v_1) = \gamma$ with $\gamma \ne \ep$ and $\eta(v_i) = \ep$ for all $i \in [n] - \{1\}$.
If $v_1$ is not a vertex of $G[C]$ then $\bar{\eta}(C) = C$ and so $\bar{\eta}(C) \in \cC$.
So we may assume that $v_1$ is a vertex of $G[C]$.
We now consider two special cases.

\begin{claim} \label{claim: degree-2, not parallel pair}
If $v_1$ has degree $2$ in $G[C]$ and the edges incident with $v_1$ do not form a cycle, then $\bar{\eta}(C) \in \cC$.
\end{claim}
\begin{proof}
There is a $2$-edge subpath $P$ of $C$ with $v_1$ as the internal vertex.
Let $v_i$ and $v_j$ be the ends of $P$.
Switching at $v_1$ replaces $P$ with a $2$-edge path $P'$ from $v_i$ to $v_j$ with $v_1$ as the internal vertex so that the $v_i$-$v_j$ walks in $P$ and $P'$ have the same $\Gamma$-value.
By Claim \ref{claim: replace a path} it follows that $\bar{\eta}(C) \in \cC$.
\end{proof}

We will use a different argument when the two edges incident with $v_1$ form a cycle.

\begin{claim} \label{claim: degree-2, parallel pair}
If $v_1$ has degree $2$ in $G[C]$ and the edges incident with $v_1$ form a cycle, then $\bar\eta(C) \in \cC$.
\end{claim}
\begin{proof}
Note that $C$ is a tight handcuff or a loose handcuff.
Let $D$ and $D'$ be the two cycles of $G[C]$ so that $v_1$ is a vertex of $D$.
Let $v_i$ be a degree-$2$ vertex of $D'$; we may assume that $i = 2$.
Let $e$ be an edge between $v_1$ and $v_2$ so that $e$ is in a $\psi$-balanced cycle $C_1$ of $G[C \cup e]$.
Let $f$ be an edge of $C_1$ incident with $v_1$.
By circuit elimination in $N$, $(C \cup e) - f$ contains a circuit $C_2$ of $N$.
Since $|C \cup e| - r_N(C \cup e) = 2$ and each element of $C \cup e$ is in a circuit of $N$, the circuits $C_1$ and $C_2$ form a modular pair whose union contains $C \cup e$.
Note that $v_1$ has degree $2$ in $G[C_2]$ and its two incident edges do not form a cycle.
Therefore $\bar\eta(C_2) \in \cC$ by Claim \ref{claim: degree-2, not parallel pair}.
Then $\bar\eta(C_2)$ and $\bar\eta(C_1)$ is a modular pair of circuits of $N$ so that both are in $\cC$ and their union contains $\bar\eta(C \cup e)$.
Since $\cC$ is a linear class, this implies that $\bar\eta(C) \in \cC$, as desired.
\end{proof}

We now proceed by induction on the degree of $v_1$ in $G[C]$ to show that $\bar\eta(C) \in \cC$.
If $v_1$ has degree at most $2$ in $G[C]$ then $\bar\eta(C) \in \cC$ by Claims \ref{claim: degree-2, not parallel pair} and \ref{claim: degree-2, parallel pair}, so we may assume that $v_1$ has degree at least $3$ in $G[C]$.
If $G[C]$ is a theta graph, then by Claim \ref{claim: replace a path} and the fact that $n \ge 3$ we may assume that $G[C]$ has at least three vertices.
Then, regardless of whether $G[C]$ is a theta or a handcuff,
$G[C]$ has a pair of vertices $v_i$ and $v_j$ with $i,j \ne 1$ and an edge $e$ between $v_i$ and $v_j$ so that $e$ is in a $\psi$-balanced cycle $C_1$ of $G[C \cup e]$ with $v_1 \in V(C_1)$.
Let $f$ be an edge of $C_1$ incident with $v_1$.
By circuit elimination for $N$, $(C \cup e) - f$ contains a circuit $C_2$ of $N$.
Since $|C \cup e| - r_N(C \cup e) = 2$ and each element of $C \cup e$ is in a circuit of $N$, the circuits $C_1$ and $C_2$ form a modular pair whose union contains $C \cup e$.
Note that $v_1$ has smaller degree in $G[C_2]$ than in $G[C]$.
By induction, $\bar\eta(C_2) \in \cC$.
Then $\bar\eta(C_2)$ and $\bar\eta(C_1)$ is a modular pair of circuits of $N$ so that both are in $\cC$ and their union contains $\bar\eta(C \cup e)$.
Since $\cC$ is a linear class, this implies that $\bar\eta(C) \in \cC$, as desired.
\end{proof}

Note that Lemma \ref{lem: switching action} is not always true when $n = 2$, because if $|\cC| = 1$ then $\cC$ is not invariant under switching when $|\Gamma| \ge 3$.
We can now prove the following restatement of Theorem \ref{thm: main converse}.

\begin{theorem} \label{thm: the converse}
Let $\Gamma$ be a finite group with a normal subgroup $\Gamma_1$ and let $n \ge 4$ be an integer.
Let $M$ be an elementary lift of the frame matroid $F(K_n^{\Gamma}, \psi_n^{\Gamma}/\Gamma_1)$ corresponding to a linear class $\cC$ so that 
a cycle of $K_n^{\Gamma}$ is in $\cC$ if and only if it is $\psi_n^{\Gamma}$-balanced.
Then there is a partition $\{\Gamma_1\} \cup \cA$ of $\Gamma$ so that if $A \in \cA$ then $A$ is malnormal and every conjugate of $A$ is in $\cA$, and $M = M(\Gamma, \Gamma_1, \cA, K_n^{\Gamma}, \psi_n^{\Gamma})$.
\end{theorem}
\begin{proof}
We will write $\circ$ and $\ep$ for the operation and identity, respectively, of $\Gamma$.
We will write $N$ for the underlying frame matroid $F(K_n^{\Gamma}, \psi_n^{\Gamma}/\Gamma_1)$ and $\psi$ for $\psi_n^{\Gamma}$.
If $\Gamma_1 = \Gamma$ then $N$ is the graphic matroid of $K_n^{\Gamma}$ and $M$ is the lift matroid of $(K_n^{\Gamma}, \psi_n^{\Gamma}/\Gamma_1)$.
If $\cA = \varnothing$, then $M(\Gamma, \Gamma_1, \cA, K_n^{\Gamma}, \psi_n^{\Gamma})$ is also the lift matroid of $(K_n^{\Gamma}, \psi_n^{\Gamma})$ by Proposition \ref{prop: special cases frame and lift matroids}, so the theorem holds with $\cA = \varnothing$.
So we may assume that $\Gamma_1 \ne \Gamma$.
This means that $\Gamma/\Gamma_1$ is nontrivial, so $r(N) = n$.
Suppose that $r(M) = n$.
Then $M = N$, and every circuit of $N$ is in $\cC$.
If $\Gamma_1$ is nontrivial, then let $\gamma \in \Gamma_1$.
Then $\{\ep_{12}, \gamma_{12}\}$ is a circuit of $N$ (a $(\psi/\Gamma_1)$-balanced cycle) that is not in $\cC$ (because it is not $\psi$-balanced), a contradiction.
Therefore $\Gamma_1$ is trivial and the theorem holds with $\cA = \{\Gamma\}$.
So we may assume that $r(M) = n + 1$.

For $\alpha \in \Gamma$ we will write $E_{\alpha}$ for $\{\alpha_{ij} \colon 1 \le i < j \le n\}$, and for a set $X \subseteq \Gamma$ we will write $E_X$ for $\cup_{\alpha \in X} E_{\alpha}$.
Note that $r_M(E_{\ep}) = r_N(E_{\ep}) = n - 1$ because all cycles in $E_{\ep}$ are $\psi$-balanced.
Using $\psi$-balanced cycles we can show that for all $\alpha \in \Gamma - \Gamma_1$, every circuit of $N$ contained in $E_{\{\ep, \alpha\}}$ is in $\cC$.

\begin{claim} \label{claim: sim is reflexive}
For all $\alpha \in \Gamma - \Gamma_1$ and $1 \le i < j \le n$ we have $r_M(E_{\ep} \cup \alpha_{ij}) = r_M(E_{\{\ep, \alpha\}}) = n$.
\end{claim}
\begin{proof}
We may assume that $(i,j) = (1,2)$.
Let $B = \{\ep_{1j} \colon 2 \le j \le n\} \cup \alpha_{12}$.
This is a basis of $N$ and is therefore independent in $M$, so $r_M(E_{\ep} \cup \alpha_{12}) = n$ since $r_M(E_{\ep}) = n - 1$.
For each $j$ with $3 \le j \le n$, the $\psi$-balanced cycle $\{\alpha_{12}, \ep_{2j}, \alpha_{1j}\}$ shows that $\alpha_{1j} \in \cl_M(E_{\ep} \cup \alpha_{12})$.
Then for all $2 \le i < j \le n$, the $\psi$-balanced cycle $\{\ep_{1i}, \alpha_{ij}, \alpha_{1j}\}$ shows that $\alpha_{ij} \in \cl_M(E_{\ep} \cup \alpha_{12})$.
Therefore $E_{\alpha} \subseteq \cl_M(E_{\ep} \cup \alpha_{12})$, so the claim holds.
\end{proof}

We will now define an equivalence relation $\sim$ on $\Gamma - \Gamma_1$; this equivalence relation will lead to $\cA$.
For $\alpha, \beta \in \Gamma - \Gamma_1$ we write $\alpha \sim \beta$ if 
$r_M(E_{\{\ep,\alpha, \beta\}}) = n$.
Equivalently, $\alpha \sim \beta$ if every circuit of $N$ contained in $E_{\{\ep,\alpha, \beta\}}$ is in $\cC$.
Clearly $\sim$ is symmetric, and $\sim$ is reflexive by Claim \ref{claim: sim is reflexive}.
To show that $\sim$ is transitive, suppose $\alpha \sim \beta$ and $\beta \sim \Gamma$ for $\alpha, \beta, \gamma \in \Gamma - \Gamma_1$.
Then $E_{\{\ep, \beta\}}$ spans $E_{\alpha}$ in $M$ because $\alpha \sim \beta$, and $E_{\{\ep, \beta\}}$ spans $E_{\gamma}$ in $M$ because $\beta \sim \gamma$.
Therefore $r_M(E_{\{\ep, \alpha, \beta, \gamma\}}) = r_M(E_{\{\ep, \beta\}}) = n$, so $\alpha \sim \gamma$.
So $\sim$ is an equivalence relation on $\Gamma - \Gamma_1$.
Let $\cA_1$ be the set of equivalence classes of $\sim$, and let $\cA = \{A \cup \ep \colon A \in \cA_1\}$.

\begin{claim} \label{claim: subgroup}
Each $A \in \cA$ is a subgroup of $\Gamma$.
\end{claim}
\begin{proof}
Let $A \in \cA$ and let $\alpha, \beta \in A - \{\ep\}$.
We will first show that $\alpha^{-1} \in A$.
The $\psi$-balanced cycle $\{\alpha^{-1}_{12}, \alpha_{23}, \ep_{13}\}$ shows that $\alpha^{-1}_{12} \in \cl_M(E_{\{\ep,\alpha\}})$.
It then follows from Claim \ref{claim: sim is reflexive} that $E_{\{\ep,\alpha\}}$ spans $E_{\alpha^{-1}}$ in $M$.
So $r_M(\{\ep, \alpha, \alpha^{-1}\}) = n$, so $\alpha \sim \alpha^{-1}$ and therefore $\alpha^{-1} \in A$.

We will next show that $\alpha \circ \beta \in A$.
The $\psi$-balanced cycle $\{(\alpha \circ \beta)_{12}, \beta^{-1}_{23}, \alpha^{-1}_{13}\}$ shows that $(\alpha \circ \beta)_{12} \in \cl_M(E_{A})$, and it follows from Claim \ref{claim: sim is reflexive} that $E_{\alpha \circ \beta} \subseteq \cl_M(E_{A})$.
So $r_M(E_{A \cup \{\alpha \circ \beta\}}) = n$ and therefore $\alpha \circ \beta \in A$ by the definition of $\sim$.
\end{proof}

We have now shown that $\cA$ is a collection of subgroups of $\Gamma$ so that each pair intersects in only the identity and each shares only the identity with $\Gamma_1$.
Note that since each $A \in \cA$ satisfies $r_M(E_A) = r_N(E_A) = n$, all circuits of $N$ contained in $E_A$ are in $\cC$.
We will next show that each $A \in \cA$ is malnormal.

\begin{claim} \label{claim: malnormal}
Each $A \in \cA$ is a malnormal subgroup of $\Gamma$.
\end{claim}
\begin{proof}
Suppose towards a contradiction that there is some $A \in \cA$, some $\alpha \in A - \{\ep\}$, and some $\gamma \in \Gamma - A$ so that $\gamma^{-1} \circ \alpha \circ \gamma \in A$.
Let $\beta = \gamma^{-1} \circ \alpha \circ \gamma$, so $\alpha \sim \beta$.
Let $C = \{\ep_{12}, \alpha_{12}, \ep_{23}, \ep_{34}, \beta_{34}\}$, so $C$ is a circuit of $N$ and $C \in \cC$.
Consider the set $C \cup \gamma_{23}$.
Since $\gamma_{23} \in \cl_N(C)$ we see that $|C \cup \gamma_{23}| - r_N(C \cup \gamma_{23}) = 2$.
Let $C' = (C - \ep_{23}) \cup \gamma_{23}$ and note that $C'$ is a circuit of $N$.
We claim that $C' \in \cC$.
Let $\eta \colon [n] \to \Gamma$ be the switching function so that $\eta(v_1) = \eta(v_2) = \gamma$ and $\eta(v_i) = \ep$ if $i \notin \{1,2\}$.
By Lemma \ref{lem: switching action} we see that $\bar\eta(C')$ is a circuit of $N$ and that $\bar\eta(C') \in \cC$ if and only if $C' \in \cC$.
Note that $\psi(\bar\eta(C'))$ is contained in $A$.
Since every cycle contained in $E_A$ is in $\cC$, this implies that $\bar\eta(C') \in \cC$, and therefore $C' \in \cC$.
Since $|C \cup \gamma_{23}| - r_N(C \cup \gamma_{23}) = 2$ and $C \cup \gamma_{23}$ contains the circuits $C, C' \in \cC$, every circuit of $N$ contained in $C \cup \gamma_{23}$ is in $\cC$ because $\cC$ is a linear class.
If $\gamma \in \Gamma_1$ then $\{\ep_{23}, \gamma_{23}\}$ is a circuit of $N$ that is in $\cC$ even though it is not $\psi$-balanced, a contradiction.
So $\gamma \notin \Gamma_1$.
Then $\{\ep_{12}, \alpha_{12}, \ep_{23}, \gamma_{23}\}$ is a tight handcuff circuit of $N$ that is in $\cC$, and is therefore a circuit of $M$.
Then $\gamma_{23} \in \cl_M(E_A)$, and so Claim \ref{claim: sim is reflexive} implies that $r_M(E_{A \cup \gamma}) = n$.
By the definition of $\sim$ we see that $\gamma \in A$, a contradiction.
Therefore $A$ is malnormal.
\end{proof}

Let $\cC' = \cC(\Gamma, \Gamma_1, \cA, K_n^{\Gamma}, \psi_n^{\Gamma})$ and let $M' = M(\Gamma, \Gamma_1, \cA, K_n^{\Gamma}, \psi_n^{\Gamma})$ from Definitions \ref{def: the linear class} and \ref{def: the matroid}.
We will now show that $\cC = \cC'$, implying that $M = M'$.
Since a cycle of $K_n^{\Gamma}$ is in $\cC$ if and only if it is $\psi$-balanced, $\cC$ and $\cC'$ contain the same cycles.
Let $C$ be a handcuff or theta graph circuit of $N$.
Let $\eta \colon [n] \to \Gamma$ be a switching function so that $\bar{\eta}(C)$ has a $\psi$-normalized spanning tree $T$.
Let $e_1, e_2 \in C - T$.
Note that $\bar{\eta}(e_1)$ and $\bar {\eta}(e_2)$ are not in $\Gamma_1$ since $C$ is a circuit of $N$.
Consider the following statements:
\begin{enumerate}[label=$(\roman*)$]
    \item $C \in \cC$.
    \item $\bar{\eta}(C) \in \cC$.
    \item There is some $A \in \cA$ so that $E_A$ contains $\bar{\eta}(e_1)$ and $\bar {\eta}(e_2)$.
    \item $\bar{\eta}(C) \in \cC'$.
    \item $C \in \cC'$.
\end{enumerate}
We will show that statements $(i)$ and $(v)$ are equivalent by showing that each consecutive pair of statements are equivalent.
Statements $(i)$ and $(ii)$ are equivalent by Lemma \ref{lem: switching action} applied to $M$, and $(iv)$ and $(v)$ are equivalent by Lemma \ref{lem: switching action} applied to $M'$.
Statements $(iii)$ and $(iv)$ are equivalent by Lemma \ref{lem: equivalences for the linear class}.
So it suffices to show that $(ii)$ and $(iii)$ are equivalent.
If $(iii)$ holds, then $\bar{\eta}(C) \subseteq E_A$, and since all circuits of $N$ contained in $E_A$ are in $\cC$ as discussed previously, $(ii)$ holds.
Conversely, suppose that $\bar{\eta}(C) \in \cC$.
Let $A \in \cA$ so that $\bar{\eta}(e_1) \in E_A$, and let $\alpha \in \Gamma$ so that $\bar{\eta}(e_2) \in E_{\alpha}$.
Since $\bar{\eta}(C)$ is a circuit of $M$ we see that $E_{A}$ spans $\bar{\eta}(e_2)$, and it then follows from Claim \ref{claim: sim is reflexive} that $r_M(E_{A \cup \alpha}) = n$. 
By the definition of $\sim$, this implies that $\alpha \in A$, so $\bar{\eta}(e_2) \in E_A$ and $(iii)$ holds.
We have now shown that $(i)$ and $(v)$ are equivalent.
Therefore $\cC = \cC'$, so $M = M'$, as desired.
\end{proof}

We comment that the statement of Theorem \ref{thm: the converse} is false when $n = 3$ because of the remark at the end of Section \ref{sec: the construction} and the fact that $K_3^{\Gamma}$ has no vertex-disjoint unbalanced cycles.

\section{Future work} \label{sec: future work}
While we have already presented several conjectures (Conjectures \ref{conj: excluded minors infinite case}, \ref{conj: excluded minors finite case}, \ref{conj: equivalence for canonical matrices}, and \ref{conj: representable implies incidence matrix of gain function}) and one problem (Problem \ref{prob: infinite group partition properties}), we close by discussing several more interesting directions for future work.

\subsection{Universal models}
Following Kahn and Kung \cite{Kahn-Kung1982}, a \emph{sequence of universal models} for a class $\cM$ of matroids is a sequence $M_2, M_3, \dots$ of simple matroids in $\cM$ so that for each $r \ge 2$, the matroid $M_r$ has rank $r$ and every simple rank-$r$ matroid in $\cM$ is a restriction of $M_r$.
For example, projective geometries form a sequence of universal models for matroids representable over a given finite field.
It is well-known that for every finite group $\Gamma$, the classes of $\Gamma$-frame matroids and $\Gamma$-lifted-graphic matroids each have a sequence of universal models (see \cite[page 47]{Zaslavsky1991} and \cite[Chapter 4]{Geelen-Nelson-Walsh-2024}).
This should be true more generally for $\cM(\Gamma, \Gamma_1, \cA)$ as well when $\Gamma$ is a Frobenius group.

\begin{conjecture}
If $\Gamma$ is a Frobenius group with Frobenius kernel $\Gamma_1$ and set $\cA$ of Frobenius complements, then $\cM(\Gamma, \Gamma_1, \cA)$ has a sequence of universal models.
\end{conjecture}

We plan to prove this conjecture in a sequel paper, and to show that the universal models are unavoidable minors of dense elementary lifts of frame matroids known as $(\alpha, 1)$-frame matroids, which are conjectured by Geelen, Gerards, and Whittle \cite[Conjecure 7.5]{Structure_Theory} to be of fundamental importance in extremal matroid theory.

\subsection{Rank-$t$ lifts}

A matroid $M$ is a \emph{rank-$t$ lift} of a matroid $N$ if there is a matroid $K$ with a set $X$ so that $K \del X = M$, $K/X = N$, and $r(M) - r(N) = t$.
Note that an elementary lift is a rank-$t$ lift with $t \in \{0,1\}$.
For a group $\Gamma$, Theorem \ref{thm: main} takes in a $\Gamma$-graph $(G, \psi)$ and outputs an elementary lift $M$ of the frame matroid $F(G, \psi)$ so that a cycle of $G$ is a circuit of $M$ if and only if it is $\psi$-balanced.
What if we wish to construct a rank-$t$ lift with $t \ge 2$?

\begin{problem} \label{prob: rank-t lifts}
Let $t$ be an integer with $t \ge 2$.
For which finite groups $\Gamma$ with nontrivial normal subgroup $\Gamma_1$ is there a construction that takes in a $\Gamma$-gain graph $(G, \psi)$ and outputs a rank-$t$ lift of the frame matroid $F(G, \psi/\Gamma_1)$ so that a cycle of $G$ is a circuit of $M$ if and only if it is $\psi$-balanced?
\end{problem}

Our motivating example comes from the class of matroids with a representation over a field $\bF$ of the following form:
$$
\begin{tabular}{|ccc|}
\hline 
& $t$ rows & \\
\hline 
& & \\
& $\bF^{\times}$-frame matrix & \\ [0.5cm]
\hline
\end{tabular}.
$$
When $\bF$ is finite, this class of matroids is important in matroid structure theory for the same reasons as outlined in Section \ref{sec: the representable case} in the case that $t = 1$.
Each column of the matrix is naturally associated with a pair $(a,b)$ where $a \in \bF^t$ and $b \in \bF^{\times}$, so it should be possible to construct the vector matroid from a $(\bF^t \rtimes \bF^{\times})$-gain graph where $\bF^{\times}$ acts on $\bF^t$ by scaling.
So we expect that $(\bF^t \rtimes \bF^{\times})$ is a valid choice for $\Gamma$ in Problem \ref{prob: rank-t lifts}, but it is unclear whether or not such a construction is also possible for other groups.
Walsh \cite[Theorem 2]{Walsh2022} and Bernstein and Walsh \cite[Theorem 2]{Bernstein-Walsh-2024} 
gave a construction for rank-$t$ lifts of graphic matroids with $t \ge 2$, so one approach for Problem \ref{prob: rank-t lifts} would be to somehow combine the work of Bernstein and Walsh with Theorem \ref{thm: main} of the present paper.

\subsection{Unified matroids in applied discrete geometry}

As a final direction for future work, we discuss the possibility of combining Theorem \ref{thm: main} with work of Tanigawa \cite{Tanigawa-2015}.
For any group $\Gamma$, Tanigawa generalizes the classes of $\Gamma$-frame matroids \cite[Section 4]{Tanigawa-2015} and $\Gamma$-lifted-graphic matroids \cite[Section 8]{Tanigawa-2015} by perturbing the rank function via a  submodular function $\mu \colon 2^{\Gamma} \to \bR_+$ called a \emph{symmetric polymatroidal function} over $\Gamma$.
In Section 9 he discusses the possibility of unifying his two constructions, and this may now be possible via Theorem \ref{thm: main}.
In particular, if $M \in \cM(\Gamma, \Gamma_1, \cA)$, it may be possible to perturb the rank function of $M$ (given by Proposition \ref{prop: rank function}) via two symmetric polymatroidal functions simultaneously, one over $\Gamma_1$ and one over $\Gamma/\Gamma_1$.
Since Tanigawa presents many different applications of his constructions in rigidity theory \cite[Section 1.1]{Tanigawa-2015}, this potential generalization may have further applications in rigidity theory.

\section*{Acknowledgments}
The author thanks Huajun Huang for helpful discussions relating to Problem \ref{prob: infinite group partition properties}, Jim Geelen and Peter Nelson for discussions on gain graphs over semidirect products, and John David Clifton for comments on an earlier draft of the paper.

\bibliographystyle{abbrv}
\bibliography{Matroid_structure}

\end{document}